\documentclass[a4paper,12pt,reqno]{amsart}
\usepackage{amsfonts}
\usepackage{amsmath}
\usepackage{amssymb}
\usepackage[a4paper]{geometry}
\usepackage{mathrsfs}
\usepackage{csquotes}
\usepackage{enumitem}
\usepackage[colorlinks]{hyperref}
\renewcommand\eqref[1]{(\ref{#1})} %Need with hyperref
\numberwithin{equation}{section}
\theoremstyle{plain}
\newtheorem{thm}{Theorem}[section]
\newtheorem{prop}[thm]{Proposition}

\theoremstyle{definition}

\newtheorem{rem}[thm]{Remark}

\renewcommand{\wp}{\mathfrak S}
\newcommand{\Rn}{\mathbb R^{n}}

\def\R{\mathcal R}

\def\e[#1]{{\textrm{e}}^{#1}}

\def\Rn{{\mathbb R}^n}
\def\G{{\mathbb G}}

\def\re{\mathbb{R}}

\def\tc{\textcolor}
\def\({\left(}
\def\){\right)}
\def\[{\left[}
\def\]{\right]}

\def\ep{\varepsilon}

\begin{document}

   \title[Unified weighted Hardy inequality]
   {Unified weighted Hardy-type inequalities}

   \author[D. Suragan]{Durvudkhan Suragan}
\address{
	Durvudkhan Suragan:
	\endgraf
	Department of Mathematics
	\endgraf
Nazarbayev University
	\endgraf
	Astana, Kazakhstan
	\endgraf
	{\it E-mail address} {\rm durvudkhan.suragan@nu.edu.kz}
}
\author[N. Yessirkegenov]{Nurgissa Yessirkegenov}
\address{
  Nurgissa Yessirkegenov:
    \endgraf
  Suleyman Demirel University
  \endgraf
  Kaskelen, Kazakhstan
  \endgraf
  and
  \endgraf
  Institute of Mathematics and Mathematical Modeling
  \endgraf
 Almaty, Kazakhstan
  \endgraf
  {\it E-mail address} {\rm nurgissa.yessirkegenov@gmail.com}
  }

\thanks{This research is funded by the Committee of Science of the Ministry of Science and Higher Education of the Republic of Kazakhstan (Grant No. AP19674900) and by the Nazarbayev University grant
20122022FD4105.}  


     \begin{abstract} We present a unified approach to obtain Hardy-type inequalities in the context of nilpotent Lie groups with sharp constants. The unified methodology employed herein allows for exploration of the sharp Hardy inequalities on various Lie group structures and improves previously known inequalities even in the classical Euclidean setting. By leveraging this framework, as an application, we derive new Caffarelli-Kohn-Nirenberg type inequalities both in the classical Euclidean setting and on general homogeneous Lie groups. 
     \end{abstract}

  \maketitle

\section{Introduction and main results}
\label{SEC:intro}
There are two main multidimensional versions of the Hardy inequality: 

\begin{itemize}
    \item the classical Hardy inequality with an interior singularity
\begin{equation}
\label{H_p}
\left( \frac{n-a}{p} \right)^{p} \int_{B(0,R)} \frac{|f|^{p}}{|x|_{E}^a} dx \le \int_{B(0,R)} \frac{\left| \nabla f\right|^p}{|x|_{E}^{a-p}} dx, \quad n \ge 2, \quad 1 < p < \infty, \quad a < n,
\end{equation}
%holds for all $f \in W^{1,p}_0(B(0,R))$, where $W_0^{1,p}(B(0,R))$ is the completion of $C_c^{\infty}(B(0,R))$ with respect to the norm $\| \nabla (\cdot )\|_{L^p(B(0,R))}$. 
for all $f \in C_0^1 \left(B(0,R)\backslash\{0\}\right)$;
\item the geometric Hardy type inequality with a boundary singularity
\begin{align}\label{H_p geo}
\(  \frac{b -1}{p}  \)^{p} \int_{B(0,R)} \frac{|f|^p}{ (R- |x|_{E})^{b}} \,dx
\le \int_{B(0,R)} \frac{|\nabla f|^p}{ (R- |x|_{E})^{b -p}} dx, \quad 1< p< \infty, \quad b >1,
\end{align}
for all $f \in C_0^1 \left(B(0,R)\backslash\{0\}\right)$.
\end{itemize}
Here $|\cdot|_{E}$ is the usual Euclidean distance and $B(0,R)$ is the $n$-dimensional Euclidean ball centred at the origin with radius $R$.

Recall that the best constant in the inequality \eqref{H_p} plays an important role to investigate stability of solution, instantaneous blow-up solution and global-in-time solution to elliptic and parabolic partial differential equations, see e.g. \cite{BV97,BG84}. The geometric Hardy inequality also holds for general bounded domain, and its best constant depends on the geometry of a given domain, see e.g. \cite{BM97,BFT03,BT06}.

Our aim in this paper to obtain Hardy inequalities in a unified way in the setting of homogeneous Lie groups with best constants. This unified approach allows for exploration of the Hardy inequalities on various group structures since the class of homogeneous Lie groups covers those of graded Lie groups, Carnot groups and Heisenberg type groups. Our results encompass previously known results and provide some novel findings, even in the specific case of the Euclidean setting. 

Note that the class of homogeneous Lie groups represents the most comprehensive subclass within the category of nilpotent Lie groups. To the best of our knowledge, there exists no example of a nilpotent Lie group except dimension nine that is not homogeneous. While the class of homogeneous Lie groups nearly encompasses the entire class of nilpotent Lie groups, it is not identical to it, as demonstrated by the existence of Dyer's nine-dimensional nilpotent Lie group, which is non-homogeneous.

Let us recall that a connected simply connected Lie group $\mathbb G$ is called a {\em homogeneous Lie group} if
its Lie algebra $\mathfrak{g}$ is equipped with a family of dilations
$$D_{\lambda}={\rm Exp}(A \,{\rm ln}\lambda)=\sum_{k=0}^{\infty}
\frac{1}{k!}({\rm ln}(\lambda) A)^{k}.$$
Here $A$ is a diagonalisable positive linear operator on the Lie algebra $\mathfrak{g}$,
and every dilation $D_{\lambda}$ satisfies 
$$\forall X,Y\in \mathfrak{g},\, \lambda>0,\;
[D_{\lambda}X, D_{\lambda}Y]=D_{\lambda}[X,Y],$$
that is, every $D_{\lambda}$ is a morphism of the Lie algebra $\mathfrak{g}$.
Then, in particular, we have
\begin{equation}
|D_{\lambda}(S)|=\lambda^{Q}|S| \quad {\rm and}\quad \int_{\mathbb{G}}f(\lambda x)
dx=\lambda^{-Q}\int_{\mathbb{G}}f(x)dx,
\end{equation}
where $Q := {\rm Tr}\,A$ is a homogeneous dimension of $\mathbb G$.
Here $dx$ is the Haar measure on the homogeneous Lie group $\mathbb{G}$ and $|S|$ is the volume of a measurable set $S\subset \mathbb{G}$. We recall that the Haar measure on a homogeneous Lie group $\mathbb{G}$ is the standard Lebesgue measure for $\Rn$ (see, for example, \cite{FS-book}).

Let $|\cdot|$ be a homogeneous quasi-norm on homogeneous Lie groups $\mathbb G$: it satisfies the usual properties of the norm except that the triangle inequality may hold with a constant $\geq 1$, see the recent book \cite{RS_book} for a detailed discussion.

If we fix a basis $\{X_{1},\ldots,X_{n}\}$ of a Lie algebra $\mathfrak{g}$
such that
$$AX_{k}=\nu_{k}X_{k}$$
for every $k$, then the matrix $A$ can be taken to be
$A={\rm diag} (\nu_{1},\ldots,\nu_{n})$.
Then each $X_{k}$ is homogeneous of degree $\nu_{k}$. By a decomposition of ${\exp}_{\mathbb{G}}^{-1}(x)$ in $\mathfrak g$, we define the vector
$$e(x)=(e_{1}(x),\ldots,e_{n}(x))$$
by the formula
$${\exp}_{\mathbb{G}}^{-1}(x)=e(x)\cdot \nabla\equiv\sum_{j=1}^{n}e_{j}(x)X_{j},$$
where $\nabla=(X_{1},\ldots,X_{n})$.
It gives the following equality
$$x={\exp}_{\mathbb{G}}\left(e_{1}(x)X_{1}+\ldots+e_{n}(x)X_{n}\right).$$
Let $\wp:=\{x\in \mathbb{G}:\,|x|=1\}$. By homogeneity and denoting $x=ry,\,y\in \wp,$ we get
$$
e(x)=e(ry)=(r^{\nu_{1}}e_{1}(y),\ldots,r^{\nu_{n}}e_{n}(y)).
$$
So one obtains
\begin{equation}\label{dfdr0}
\frac{d}{d|x|}(f(x))=\frac{d}{dr}(f(ry))=
 \frac{d}{dr}(f({\exp}_{\mathbb{G}}
\left(r^{\nu_{1}}e_{1}(y)X_{1}+\ldots
+r^{\nu_{n}}e_{n}(y)X_{n}\right))).
\end{equation}
Throughout this paper, we use the notation
\begin{equation}\label{EQ:Euler}
\mathcal{R}_{|\cdot|} :=\frac{d}{dr},
\end{equation}
that is,
\begin{equation}\label{dfdr}
	\frac{d}{d|x|}(f(x))=\mathcal{R}_{|x|}f(x), \quad\forall x\in \mathbb G,
\end{equation}
for a homogeneous quasi-norm $|x|$ on the homogeneous Lie group $\mathbb G$. Note that the identity element of $\mathbb G$ is the origin of $\mathbb{R}^{n}$, so simply we denote it by $0$. Thus, throughout this paper we denote by $B(0,R)$ a quasi-ball of $\mathbb{G}$ centred at $0$ with radius $R$ with respect to the quasi-norm $|\cdot|$ of $\mathbb{G}$.

Thus, our first result is as follows:
%%%%%%%%%%%%%%%%%%%%%%%%%%%
\begin{thm}\label{T IH remainder} Let $\mathbb{G}$ be a homogeneous Lie group
of homogeneous dimension $Q$. Let $|\cdot|$ be a homogeneous quasi-norm. Let $B(0,R)$ be a quasi-ball of $\mathbb{G}$ with radius $R$ with respect to $|\cdot|$. Let $1 < p<\infty$, $b>1$, $a< Q$ and $0<c\leq \frac{Q-a}{b-1}$. Then we have
\begin{multline}\label{IH remainder}
\( \frac{b -1}{p} c \)^{p} \int_{B(0,R)} \frac{|f|^{p}}{|x|^a \( 1- \( \frac{|x|}{R} \)^c \)^b } dx + \psi_{Q,p,a, b} (f) \\ \le \int_{B(0,R)} \frac{\left| \R_{|x|} f \right|^p}{|x|^{a-p}\( 1-\( \frac{|x|}{R} \)^c \)^{b-p}} dx
\end{multline}
for all $f \in C_0^1(B(0,R)\backslash\{0\})$, where 
\begin{align*}
\psi_{Q,p, a, b} (f ) = (Q-a - (b -1) c) \( \frac{b -1}{p} c \)^{p-1}\int_{B(0,R)} \frac{|f|^p}{|x|^{a} \(1- \( \frac{|x|}{R} \)^c \)^{b-1}} \,dx
\end{align*}
for $c<\frac{Q-a}{b-1}$, and 
\begin{multline*}
\psi_{Q,p,a, b} (f)\\
=\begin{cases} 
C \int_{B(0,R)} |x|^{p-Q}\( 1- \( \frac{|x|}{R} \)^c \)^{p-1} \left| \R_{|x|} \( \frac{f(x)}{\( \( \frac{|x|}{R} \)^{-c} -1 \)^{\frac{b -1}{p}}} \) \right|^p \,dx, \;p \in [2, \infty), \\
C \( \int_{B(0,R)} |x|^{p-Q}\( 1- \( \frac{|x|}{R} \)^c \)^{p-1} \left| \R_{|x|} \( \frac{f(x)}{\( \( \frac{|x|}{R} \)^{-c} -1 \)^{\frac{b -1}{p}}} \) \right|^p \,dx \)^{\frac{2}{p}} &\tc{white}{a}\\
\hspace{5em}\times \( \int_{B(0,R)} \frac{\left| |\R_{|x|} f|+\frac{b-1}{p}\frac{|f|}{|x|(1-|x|^{c})}\right|^p}{|x|^{a-p} \( 1-\( \frac{|x|}{R} \)^c \)^{b-p}} dx \)^{\frac{p-2}{p}}, \;p \in (1, 2),
\end{cases}
\end{multline*}
for $c=\frac{Q-a}{b-1}$.

Moreover, if we drop the non-negative remainder term $\psi_{Q,p,a, b} (f)$ in \eqref{IH remainder}, then the constant $\left(\frac{b-1}{p} c\right)^p$ is optimal and can not be attained for $f \not\equiv 0$ for which the right-hand side is finite. 
\end{thm}

\begin{rem}\label{rem1}
\begin{itemize}
\item In the special case $b=p>1$, $a<Q$ and $c=\frac{Q-a}{p-1}$, we obtain from \eqref{IH remainder} that
\begin{equation}\label{derive_clas_Har}
\begin{split}
\int_{B(0,R)} \frac{\left|\R_{|x|} f\right|^p}{|x|^{a-p}} d x&\geq \left(\frac{p-1}{p} c\right)^p \int_{B(0,R)} \frac{|f|^p}{|x|^a\left(1-\left(\frac{|x|}{R}\right)^c\right)^p} d x\\& \geq 
\left(\frac{Q-a}{p}\right)^p \int_{B(0,R)} \frac{|f|^p}{|x|^a} d x.
\end{split}
\end{equation}
For the weighted $L^{p}$-Hardy inequality \eqref{derive_clas_Har} on a general homogeneous Lie group, we refer to \cite[Theorem 3.1]{RS17} for $a=p$ and to \cite[Theorem 3.4]{RSY18} for the general case.
    \item In the Abelian (isotropic or anisotropic) case ${\mathbb G}=(\mathbb R^{n},+)$, we have
$Q=n$, so for any quasi-norm $|\cdot|$ on $\mathbb R^{n}$, the inequality \eqref{IH remainder} after dropping the non-negative term $\psi_{Q,p,a,b}(f)$ implies 
	\begin{equation}\label{Euc_1}
\( \frac{b -1}{p} c \)^{p} \int_{B(0,R)} \frac{|f|^{p}}{|x|^a \( 1- \( \frac{|x|}{R} \)^c \)^b } dx \le \int_{B(0,R)} \frac{\left|\frac{df}{d|x|}\right|^p}{|x|^{a-p}\( 1-\( \frac{|x|}{R} \)^c \)^{b-p}} dx,
\end{equation}	
where the constant is optimal for any homogeneous quasi-norm. In the case of the standard Euclidean distance $|x|_{E}=\sqrt{x^{2}_{1}+\ldots+x^{2}_{n}}$ by using the Schwartz inequality from the inequality \eqref{Euc_1} we obtain 
	\begin{equation}\label{Euc_11}
\( \frac{b -1}{p} c \)^{p} \int_{B(0,R)} \frac{|f|^{p}}{|x|_{E}^a \( 1- \( \frac{|x|_{E}}{R} \)^c \)^b } dx \le \int_{B(0,R)} \frac{\left| \nabla f\right|^p}{|x|_{E}^{a-p}\( 1-\( \frac{|x|_{E}}{R} \)^c \)^{b-p}} dx,
\end{equation}
where $\nabla$ is the standard gradient on $\mathbb R^{n}$. Similarly arguing as \eqref{derive_clas_Har}, the inequality \eqref{Euc_11} implies the classical Hardy inequality \eqref{H_p} when $b=p>1$, $a<n$ and $c=\frac{n-a}{p-1}$.
\item On the other hand, when $b>1$, $a=p$ and $c=1 \leq \frac{n-p}{b-1}$, the inequality \eqref{Euc_11} implies the geometric Hardy inequality \eqref{H_p geo}:
$$
\begin{aligned}
& \left(\frac{b-1}{p}\right)^p \int_{B(0,R)} \frac{|f|^p}{(R-|x|_{E})^b} d x \\
& \leq\left(\frac{b-1}{p}\right)^p R^{p-b} \int_{B(0,R)} \frac{|f|^p}{|x|_{E}^p\left(1-\frac{|x|_{E}}{R}\right)^b} d x \\
& \leq R^{p-b} \int_{B(0,R)} \frac{\left|\nabla f\right|^p}{\left(1-\frac{|x|_{E}}{R}\right)^{b-p}} d x=\int_{B(0,R)} \frac{\left|\nabla f\right|^p}{\left(R-|x|_{E}\right)^{b-p}} d x.
\end{aligned}
$$
\item In the case when $a=b=c=p=2\leq n-2$ and $R=1$, the inequality \eqref{Euc_1} with the Euclidean distance $|x|_{E}$ gives the Hardy-Sobolev-Maz'ya inequality, see e.g. \cite[Corollary 3]{Maz11}, \cite{BFL08,TT07}. In the Euclidean case, we also refer to \cite{Iok19} for the special case of \eqref{IH remainder} without the non-negative term $\psi_{Q,p,a,b}(f)$ where $a = n-(b -1)c$ and $b = p$, and to \cite{San21} for the general case. Here, in the Abelian case $\G=(\Rn,+)$ and $Q=n$ with the Euclidean distance $|\cdot|_{E}$, the inequality \eqref{IH remainder} improves the recent result in \cite[Theorem 2.1]{San21} for $1<p<2$. We also refer to the recent work \cite{ACFM22} for the Hardy inequality on domains with a geometric
boundary condition. On the Heisenberg group, we refer to \cite[Theorem 1.3]{Yan13} and \cite{Amb05} for similar results with the Carnot-Carath\'{e}odory distance.  
\item Note that using $1-r^c = c \log \frac{1}{r} + o(c)$ as $c \to 0$ one can take a limit of \eqref{IH remainder} as $c \to 0$ after dropping the non-negative term $\psi_{Q,p,a,b}(f)$ to obtain 
\begin{equation}\label{IH log}
\( \frac{b -1}{p} \)^{p} \int_{B(0,R)} \frac{|f|^{p}}{|x|^a \( \log \frac{R}{|x|} \)^b } dx \le \int_{B(0,R)} \frac{\left| \R_{|x|} f \right|^p}{|x|^{a-p}\( \log \frac{R}{|x|} \)^{b-p}} dx.
\end{equation}
For the special case of \eqref{IH log} we refer to \cite[Part 2, Theorem 2.12]{Amb05} on the Heisenberg group, and to \cite{RSY18} and \cite{RS_book} on homogeneous Lie groups as well as to \cite{Tak15, MOW15} for the Euclidean versions. 
\end{itemize}
\end{rem}

In the case $p=2$, we have the following sharp remainder formula for the unified Hardy inequality \eqref{IH remainder}: 
\begin{thm}
Let $\mathbb{G}$ be a homogeneous Lie group of homogeneous dimension $Q$ and let $|\cdot|$ be an arbitrary homogeneous quasi-norm. Let $B(0,R)$ be a quasi-ball of $\mathbb{G}$ with radius $R$ with respect to $|\cdot|$. Let $b>1$, $a<Q$ and $0<c\leq \frac{Q-a}{b-1}$. Then we have  
\begin{multline}\label{eq_L2_iden}
 \left(\frac{b-1}{2} c\right)^2 \int_{B(0,R)} \frac{|f|^2}{|x|^a\left(1-\left(\frac{|x|}{R}\right)^c\right)^b} d x=\int_{B(0,R)} \frac{\left|\mathcal{R}_{|x|} f\right|^2}{|x|^{a-2}\left(1-\left(\frac{|x|}{R}\right)^c\right)^{b-2}} d x \\
 -\frac{((b-1) c)^2}{2} \int_{B(0,R)}\left|f+\frac{2}{(b-1) c}|x|\left(1-\left(\frac{|x|}{R}\right)^c\right)\left(\mathcal{R}_{|x|} f\right)\right|^2 
 \frac{d x}{|x|^a\left(1-\left(\frac{|x|}{R}\right)^c\right)^b} \\
 -(Q-a-(b-1) c)\left(\frac{b-1}{2} c\right) \int_{B(0,R)} \frac{|f|^2}{|x|^a\left(1-\left(\frac{|x|}{R}\right)^c\right)^{b-1}} d x 
\end{multline}
for all complex-valued functions $f \in C_0^{1}\left(B(0,R)\backslash\{0\}\right)$.
\end{thm}
\begin{rem} In the special case $c=\frac{Q-a}{b-1}$ with $a=2 \alpha+2$ the identity \eqref{eq_L2_iden} implies \cite[Theorem 4.1]{RS17}. 
\end{rem}

The following result is the $L^{p}$-version of the identity \eqref{eq_L2_iden} for real-valued functions $f \in$ $C_0^{1}\left(B(0,R)\backslash\{0\}\right)$:
\begin{thm}\label{equality_Hardy} Let $\mathbb{G}$ be a homogeneous Lie group
of homogeneous dimension $Q$. Let $|\cdot|$ be a homogeneous quasi-norm. Let $B(0,R)$ be a quasi-ball of $\mathbb{G}$ with radius $R$ with respect to $|\cdot|$. Let $1 < p<\infty$, $b>1$, $a< Q$ and $0<c\leq \frac{Q-a}{b-1}$. Then we have 
\begin{equation}\label{eq_Lp_iden}
\begin{aligned}
&\left(\frac{b-1}{p} c\right)^{p} \int_{B(0,R)} \frac{|f|^{p}}{|x|^{a}\left(1-\left(\frac{|x|}{R}\right)^{c}\right)^{b}} d x=\int_{B(0,R)} \frac{\left|\R_{|x|} f\right|^{p}}{|x|^{a-p}\left(1-\left(\frac{|x|}{R}\right)^{c}\right)^{b-p}} d x \\
&-p\left(\frac{b-1}{p} c\right)^{p} \int_{B(0,R)} I_{p}(f, u) \frac{|f-u|^{2}}{|x|^{a}\left(1-\left(\frac{|x|}{R}\right)^{c}\right)^{b}}d x \\
&-(Q-a-(b-1) c)\left(\frac{b-1}{p} c\right)^{p-1} \int_{B(0,R)} \frac{|f|^{p}}{|x|^{a}\left(1-\left(\frac{|x|}{R}\right)^{c}\right)^{b-1}} d x
\end{aligned}
\end{equation}
for all real-valued functions $f \in C_{0}^{1}\left(B(0,R)\backslash\{0\}\right)$, where
\begin{equation}\label{def_func}
I_p(f, u)=(p-1) \int_0^1|\xi f+(1-\xi) u|^{p-2} \xi d \xi
\end{equation}
and $$u=\frac{p}{(b-1) c}|x|\left(1-\left(\frac{|x|}{R}\right)^{c}\right)\left(-\R_{|x|}f\right).$$
\end{thm}

%$|A|$ denotes the Lebesgue measure of a set $A \subset \re^N$ and 

%The Schwarz symmetrization $f^{\#} \colon \re^N \to [0, \infty]$ of $f$ is given by 
%\begin{align*}
%f^{\#}(x) = f^{\#}(|x|)=\inf \left\{ \tau >0 \,: \, | \{ y \in \re^N \, :\, |f(y)| > \tau \} \,| \le |\,B_{|x|} \, | \right\}.
%\end{align*}
%Throughout the paper, if a radial function $f$ is written as $f(x) = \tilde{f}(|x|)$ by some function $\tilde{f} = \tilde{f}(r)$, we write $f(x)= f(|x|)$ with admitting some ambiguity.

%%%%%%%%%%%%%%%%%%%%%%%%%%%%%%%%%

% \S 2 Generalizations of the improved Hardy inequality

%%%%%%%%%%%%%%%%%%%%%%%%%%%%%%%%%

By iterating unified Hardy identities \eqref{eq_L2_iden} and \eqref{eq_Lp_iden} we obtain the following two theorems for unified higher-order $L^2$ and $L^p$-Hardy type identities on general homogeneous Lie groups, respectively:

\begin{thm}\label{thm_high_1.13}
Let $\mathbb{G}$ be a homogeneous Lie group of homogeneous dimension $Q$ and let $|\cdot|$ be an arbitrary homogeneous quasi-norm. Let $B(0,R)$ be a quasi-ball of $\mathbb{G}$ with radius $R$ with respect to $|\cdot|$. Let $b+2(k-1)>1, a+2(k-1)<Q$ and $0<c\leq \frac{Q-a-2(k-1)}{b+2(k-1)-1}$. Then we have
\begin{multline}\label{eq_high_L2} 
\prod_{j=0}^{k-1}\left(\frac{b+2 j-1}{2} c\right)^2 \int_{B(0,R)} \frac{|f|^2}{|x|^{a+2(k-1)}\left(1-\left(\frac{|x|}{R}\right)^c\right)^{b+2(k-1)}} d x 
+2 \sum_{i=1}^k \prod_{j=0}^{i-1}\\\left(\frac{b+2 j-1}{2} c\right)^2 
\int_{B(0,R)}\left|\mathcal{R}_{|x|}^{k-i} f+\frac{2}{(b+2(i-1)-1) c}|x|\left(1-\left(\frac{|x|}{R}\right)^c\right)\left(\mathcal{R}_{|x|}^{k-i+1} f\right)\right|^2 
\\ \times \frac{d x}{|x|^{a+2(i-1)}\left(1-\left(\frac{|x|}{R}\right)^c\right)^{b+2(i-1)}} \\
 +\sum_{i=1}^k \prod_{j=0}^{i-1}(Q-a-2(i-1)-(b+2(i-1)-1) c)\left(\frac{b+2 j-1}{2} c\right)^2\left(\frac{b+2(i-1)-1}{2} c\right)^{-1} \\
 \int_{B(0,R)} \frac{\left|\mathcal{R}_{|x|}^{k-i} f\right|^2}{|x|^{a+2(i-1)}\left(1-\left(\frac{|x|}{R}\right)^c\right)^{b+2(i-1)-1} }d x 
 =\int_{B(0,R)} \frac{\left|\mathcal{R}_{|x|}^k f\right|^2}{|x|^{a-2}\left(1-\left(\frac{|x|}{R}\right)^c\right)^{b-2} }d x
\end{multline}
for all complex-valued functions $f \in C_0^{k}\left(B(0,R)\backslash\{0\}\right)$.
\end{thm}
\begin{thm}\label{thm_high_1.14} Let $\mathbb{G}$ be a homogeneous Lie group of homogeneous dimension $Q$ and let $|\cdot|$ be an arbitrary homogeneous quasi-norm. Let $B(0,R)$ be a quasi-ball of $\mathbb{G}$ with radius $R$ with respect to $|\cdot|$. Let $1<p<\infty$, $b+(k-1)p>1$, $a+(k-1) p<Q$ and $0<c\leq  \frac{Q-a-(k-1) p}{b+(k-1) p-1}$. Then we have
\begin{multline}\label{eq_high_Lp}
 \prod_{j=0}^{k-1}\left(\frac{b+j p-1}{p} c\right)^p \int_{B(0,R)} \frac{|f|^p}{|x|^{a+(k-1) p}\left(1-\left(\frac{|x|}{R}\right)^c\right)^{b+(k-1) p}} d x 
 +p \sum_{i=1}^k \prod_{j=0}^{i-1}\\ \left(\frac{b+j p-1}{p} c\right)^{p}
 \int_{B(0,R)} I_p(\mathcal{R}_{|x|}^{k-i} f, u) \frac{|\mathcal{R}_{|x|}^{k-i} f-u|^{2}}{|x|^{a+(i-1) p}\left(1-\left(\frac{|x|}{R}\right)^c\right)^{b+(i-1) p}}d x \\
 +\sum_{i=1}^k \prod_{j=0}^{i-1}(Q-a-(i-1) p-(b+(i-1) p-1) c)\left(\frac{b+j p-1}{p} c\right)^p\left(\frac{b+(i-1) p-1}{p} c\right)^{-1} \\
 \int_{B(0,R)} \frac{\left|\mathcal{R}_{|x|}^{k-i} f\right|^p}{|x|^{a+(i-1) p}\left(1-\left(\frac{|x|}{R}\right)^c\right)^{b+(i-1) p-1}} d x 
=\int_{B(0,R)} \frac{\left|\mathcal{R}_{|x|}^k f\right|^p}{|x|^{a-p}\left(1-\left(\frac{|x|}{R}\right)^c\right)^{b-p}} d x 
\end{multline}
for all real-valued functions $f \in C_0^{k}\left(B(0,R)\backslash\{0\}\right)$, where the functional $I_{p}$ is defined in \eqref{def_func} and
$$u=\frac{p}{(b+(i-1) p-1) c}|x|\left(1-\left(\frac{|x|}{R}\right)^c\right)\left(-\mathcal{R}_{|x|}^{k-i+1} f\right).$$
Moreover, we have 
\begin{multline}\label{eq_high_Lp_ineq}
 \prod_{j=0}^{k-1}\left(\frac{b+j p-1}{p} c\right)^p \int_{B(0,R)} \frac{|f|^p}{|x|^{a+(k-1) p}\left(1-\left(\frac{|x|}{R}\right)^c\right)^{b+(k-1) p}} d x 
 \\ 
 \leq \int_{B(0,R)} \frac{\left|\mathcal{R}_{|x|}^k f\right|^p}{|x|^{a-p}\left(1-\left(\frac{|x|}{R}\right)^c\right)^{b-p}} d x 
\end{multline}
 for all complex-valued functions $f \in C_0^{k}\left(B(0,R)\backslash\{0\}\right)$. 
\end{thm}

As applications of identities \eqref{eq_high_L2}-\eqref{eq_high_Lp} and inequality \eqref{eq_high_Lp_ineq}, we obtain the following unified higher-order Caffarelli-Kohn-Nirenberg type inequalities with remainder terms for all complex-valued and real-valued functions $f \in C_0^{k}\left(B(0,R)\backslash\{0\}\right)$:
\begin{thm}\label{thm_CKN_unif2} Let $\mathbb{G}$ be a homogeneous Lie group of homogeneous dimension $Q$ and let $|\cdot|$ be an arbitrary homogeneous quasi-norm. Let $B(0,R)$ be a quasi-ball of $\mathbb{G}$ with radius $R$ with respect to $|\cdot|$. Let $\beta, \gamma \in \mathbb{R}, b+(k-1)p>1, a+(k-1) p<Q, c \in\left(0, \frac{Q-a-(k-1) p}{b+(k-1) p-1}\right]$ and $k \in \mathbb{N}$. Let $1<p, q<\infty, 0<r<\infty$ with $p+q \geq r$ and $\delta \in[0,1] \cap\left[\frac{r-q}{r}, \frac{p}{r}\right]$. Assume that $\frac{\delta r}{p}+$ $\frac{(1-\delta) r}{q}=1$ and $\gamma=-\delta+\beta(1-\delta)$. Then for all real-valued functions $f \in C_0^{k}\left(B(0,R)\backslash\{0\}\right)$ we have
\begin{multline}\label{CKN_thm2}
 \prod_{j=0}^{k-1}\left(\frac{b+j p-1}{p} c\right)^\delta\left\|\omega^\gamma f\right\|_{L^r(B(0,R))}  \leq\left(\left\|\frac{\mathcal{R}_{|x|}^k f}{\|\left. x\right|^{\frac{a-p}{p}}\left(1-\left(\frac{|x|}{R}\right)^c\right)^{\frac{b-p}{p}}}\right\|_{L^p(B(0,R))}^p-\Re\right)^{\frac{\delta}{p}}\\ \times \left\|\omega^\beta f\right\|_{L^q(B(0,R)}^{1-\delta},
\end{multline}
where $\omega(x)=|x|^{\frac{a+(k-1) p}{p}}\left(1-\left(\frac{|x|}{R}\right)^c\right)^{\frac{b+(k-1) p}{p}}$ and the remainder term $\Re$ is defined by
$$
\begin{aligned}
& \Re=p \sum_{i=1}^k \prod_{j=0}^{i-1}\left(\frac{b+j p-1}{p} c\right)^p \\
&\int_{B(0,R)} I_p(\mathcal{R}_{|x|}^{k-i} f, u) \frac{|\mathcal{R}_{|x|}^{k-i} f-u|^{2}}{|x|^{a+(i-1) p}\left(1-\left(\frac{|x|}{R}\right)^c\right)^{b+(i-1) p}}d x \\
& +\sum_{i=1}^k \prod_{j=0}^{i-1}(Q-a-(i-1) p-(b+(i-1) p-1) c)\left(\frac{b+j p-1}{p} c\right)^p
\\
& \times \left(\frac{b+(i-1) p-1}{p} c\right)^{-1} \int_{B(0,R)} \frac{\left|\mathcal{R}_{|x|}^{k-i} f\right|^p}{|x|^{a+(i-1) p}\left(1-\left(\frac{|x|}{R}\right)^c\right)^{b+(i-1) p-1}} d x
\end{aligned}
$$
with $I_{p}$ and $u$ from Theorem \ref{thm_high_1.14}. Moreover, in the cases $p=2$ and general $p$ without the remainder term, the inequality \eqref{CKN_thm2} holds true for any complex-valued function $f \in C_0^{k}\left(B(0,R)\backslash\{0\}\right)$.
\end{thm}
\begin{rem} We refer to \cite[Theorem 7.1]{RSY18_Tran} (see also \cite{RSY17_Comp}) for the case $k=1$, $b=p>1$, $a<Q$ and $c=\frac{Q-a}{p-1}$.
\end{rem}

Let us now state other unified versions of \eqref{H_p} and \eqref{H_p geo} for the radial derivative operator on a homogeneous Lie group $\G$ with any homogeneous quasi-norm $|\cdot|$:
\begin{thm}\label{thm_gen_2ineq} Let $\mathbb{G}$ be a homogeneous Lie group
of homogeneous dimension $Q$. Let $|\cdot|$ be a homogeneous quasi-norm. Let $B(0,R)$ be a quasi-ball of $\mathbb{G}$ with radius $R$ with respect to $|\cdot|$.
Let $1 < p<\infty$, $a < Q$, $b \geq 1$ and $c>0$.
\begin{enumerate}[label=(\roman*)]
    \item Let $f \in C_0^{1}(B(0,R) \backslash\{0\})$ be a complex-valued function. Then we have
\begin{equation}\label{(7)}
 \left(\frac{Q-a}{p}\right)^{p} \int_{B(0,R)} \frac{|f|^{p}}{|x|^{a}\left(1-\left(\frac{|x|}{R}\right)^{c}\right)^{b-1}} d x \leq \int_{B(0,R)} \frac{\left|\R_{|x|} f\right|^{p}}{|x|^{a-p}\left(1-\left(\frac{|x|}{R}\right)^{c}\right)^{b-1}} d x   
\end{equation}
for all $f \in C_{0}^{1}\left(B(0,R)\backslash\{0\}\right)$. Furthermore, the constant $\left(\frac{Q-a}{p}\right)^{p}$ in \eqref{(7)} is optimal and is not attained for $f \not\equiv 0$ for which the right-hand side is finite.
\item Let $f \in C_0^{1}(B(0,R) \backslash\{0\})$ be a real-valued function. Then we have
\begin{multline}\label{(3.2)}
 \int_{B(0,R)} \frac{\left|\R_{|x|} f\right|^{p}}{|x|^{a-p}\left(1-\left(\frac{|x|}{R}\right)^{c}\right)^{b-1}} d x - \left(\frac{Q-a}{p}\right)^{p} \int_{B(0,R)} \frac{|f|^{p}}{|x|^{a}\left(1-\left(\frac{|x|}{R}\right)^{c}\right)^{b-1}} d x\\ 
 =\left(\frac{Q-a}{p}\right)^{p}\frac{(b-1)cp}{Q-a}\int_{B(0,R)} \frac{|f|^{p}}{|x|^{a-c}\left(1-\left(\frac{|x|}{R}\right)^{c}\right)^{b}} d x \\+p\left(\frac{Q-a}{p}\right)^{p} \int_{B(0,R)} I_p(v, u)|v-u|^2 d x,
\end{multline}
where the functional $I_{p}$ is defined in \eqref{def_func} with
$$
u=-\frac{p}{Q-a} \frac{\mathcal{R}_{|x|} f}{|x|^{\frac{a-p}{p}}\left(1-\left(\frac{|x|}{R}\right)^{c}\right)^{\frac{b-1}{p}}}
\;\;\text{and}\;\;
v=\frac{f}{|x|^{\frac{a}{p}}\left(1-\left(\frac{|x|}{R}\right)^{c}\right)^{\frac{b-1}{p}}}.
$$
\item In the case $p=2$, the identity \eqref{(3.2)} holds for complex-valued functions as well. Namely, if $f \in C_0^{1}(B(0,R) \backslash\{0\})$ is a complex-valued function then we have
\begin{multline}\label{(3.3)}
 \int_{B(0,R)} \frac{\left|\R_{|x|} f\right|^{2}}{|x|^{a-2}\left(1-\left(\frac{|x|}{R}\right)^{c}\right)^{b-1}} d x - \left(\frac{Q-a}{2}\right)^{2} \int_{B(0,R)} \frac{|f|^{2}}{|x|^{a}\left(1-\left(\frac{|x|}{R}\right)^{c}\right)^{b-1}} d x\\ 
 =\left(\frac{Q-a}{2}\right)^{2}\frac{2(b-1)c}{Q-a}\int_{B(0,R)} \frac{|f|^{2}}{|x|^{a-c}\left(1-\left(\frac{|x|}{R}\right)^{c}\right)^{b}} d x \\+\frac{(Q-a)^{2}}{2} \int_{B(0,R)} 
 \left|\frac{2}{Q-a} \frac{\mathcal{R}_{|x|} f}{|x|^{\frac{a-2}{2}}(1-\left(\frac{|x|}{R}\right)^{c})^{\frac{b-1}{2}}}+\frac{f}{|x|^{\frac{a}{2}}(1-\left(\frac{|x|}{R}\right)^{c})^{\frac{b-1}{2}}}\right|^2 d x.
\end{multline}
\end{enumerate}
\end{thm}
\begin{thm}\label{thm_gen_2ineq_8} Let $\mathbb{G}$ be a homogeneous Lie group
of homogeneous dimension $Q$. Let $|\cdot|$ be a homogeneous quasi-norm. Let $B(0,R)$ be a quasi-ball of $\mathbb{G}$ with radius $R$ with respect to $|\cdot|$.
Let $1 < p<\infty$, $a < Q$, $b>1$ and $0<c\leq \frac{Q-a}{p-1}$.
\begin{enumerate}[label=(\roman*)]
    \item Let $f \in C_0^{1}(B(0,R) \backslash\{0\})$ be a complex-valued function. Then we have
\begin{equation}\label{(8)}
\left(\frac{b-1}{p} c\right)^{p} \int_{B(0,R)} \frac{|f|^{p}}{|x|^{a-c}\left(1-\left(\frac{|x|}{R}\right)^{c}\right)^{b}} d x \leq \int_{B(0,R)} \frac{\left|\R_{|x|} f\right|^{p}}{|x|^{a-c+(c-1) p}\left(1-\left(\frac{|x|}{R}\right)^{c}\right)^{b-p}} d x
\end{equation}
for all $f \in C_{0}^{1}\left(B(0,R)\backslash\{0\}\right)$. Furthermore, the constant $\left(\frac{b-1}{p} c\right)^{p}$ in \eqref{(8)} is optimal and is not attained for $f \not\equiv 0$ for which the right-hand side is finite.
\item Let $f \in C_0^{1}(B(0,R) \backslash\{0\})$ be a real-valued function. Then we have
\begin{multline}\label{(3.2_8)}
 \int_{B(0,R)} \frac{\left|\R_{|x|} f\right|^{p}}{|x|^{a-p}\left(1-\left(\frac{|x|}{R}\right)^{c}\right)^{b-1}} d x - \left(\frac{b-1}{p} c\right)^{p} \int_{B(0,R)} \frac{|f|^{p}}{|x|^{a-c}\left(1-\left(\frac{|x|}{R}\right)^{c}\right)^{b}} d x\\ 
=\left(\frac{b-1}{p}c\right)^{p}\frac{(Q-a)p}{(b-1)c}\int_{B(0,R)} \frac{|f|^{p}}{|x|^{a}\left(1-\left(\frac{|x|}{R}\right)^{c}\right)^{b-1}} d x \\+p\left(\frac{b-1}{p}c\right)^{p}\int_{B(0,R)} I_p(v, u)|v-u|^2 d x,
\end{multline}
where the functional $I_{p}$ is defined in \eqref{def_func} with
$$
u=-\frac{p}{(b-1)c} \frac{\mathcal{R}_{|x|} f}{|x|^{\frac{a-p+c(p-1)}{p}}(1-\left(\frac{|x|}{R}\right)^{c})^{\frac{b-p}{p}}}
\;\;\text{and}\;\;
v=\frac{f}{|x|^{\frac{a-c}{p}}(1-\left(\frac{|x|}{R}\right)^{c})^{\frac{b}{p}}}.
$$
\item In the case $p=2$, the identity \eqref{(3.2_8)} holds for complex-valued functions as well. Namely, if $f \in C_0^{1}(B(0,R) \backslash\{0\})$ is a complex-valued function then we have
\begin{multline}\label{(3.3_8)}
 \int_{B(0,R)} \frac{\left|\R_{|x|} f\right|^{2}}{|x|^{a-2}\left(1-\left(\frac{|x|}{R}\right)^{c}\right)^{b-1}} d x - \left(\frac{b-1}{p} c\right)^{2} \int_{B(0,R)} \frac{|f|^{2}}{|x|^{a-c}\left(1-\left(\frac{|x|}{R}\right)^{c}\right)^{b}} d x\\ 
=\left(\frac{b-1}{2}c\right)^{p}\frac{2(Q-a)}{(b-1)c}\int_{B(0,R)} \frac{|f|^{2}}{|x|^{a}\left(1-\left(\frac{|x|}{R}\right)^{c}\right)^{b-1}} d x \\+\frac{(b-1)^{2}c^{2}}{2}\int_{B(0,R)} \left|\frac{2}{(b-1)c} \frac{\mathcal{R}_{|x|} f}{|x|^{\frac{a-2+c}{2}}\left(1-\left(\frac{|x|}{R}\right)^{c}\right)^{\frac{b-2}{2}}}+\frac{f}{|x|^{\frac{a-c}{p}}\left(1-\left(\frac{|x|}{R}\right)^{c}\right)^{\frac{b}{2}}}\right|^2 d x.
\end{multline}
\end{enumerate}
\end{thm}
\begin{rem} \begin{itemize}
\item In the Abelian case $\G=(\Rn,+)$ and $Q=n$,  inequalities \eqref{(7)} and \eqref{(8)} were obtained in \cite[Theorem 1.4]{San21} when $|\cdot|_{E}$ is an Euclidean norm. Here, Theorems \ref{thm_gen_2ineq} and \ref{thm_gen_2ineq_8} introduce $L^{2}$ and $L^p$ sharp remainder formulae for those inequalities with any homogeneous quasi-norm. These are new even in the Abelian setting. 
    \item Note that the inequality \eqref{(7)} implies the classical Hardy  inequality \eqref{H_p} when $b=1$, while \eqref{(8)} implies the geometric Hardy inequality \eqref{H_p geo} when $R=a=c=1$.
    \item As in Remark \ref{rem1}, when $c \rightarrow 0$ the inequalities \eqref{(7)} and \eqref{(8)} yields, respectively,
$$
\begin{aligned}
&\left(\frac{Q-a}{p}\right)^{p} \int_{B(0,R)} \frac{|f|^{p}}{|x|^{a}\left(\log \frac{R}{|x|}\right)^{b}} d x \leq \int_{B(0,R)} \frac{\left|\R_{|x|} f\right|^{p}}{|x|^{a-p}\left(\log \frac{R}{|x|}\right)^{b}} d x \\
&\left(\frac{b-1}{p}\right)^{p} \int_{B(0,R)} \frac{|f|^{p}}{|x|^{a}\left(\log \frac{R}{|x|}\right)^{b}} d x \leq \int_{B(0,R)} \frac{\left|\R_{|x|} f\right|^{p}}{|x|^{a-p}\left(\log \frac{R}{|x|}\right)^{b-p}} d x.
\end{aligned}
$$
\end{itemize}
\end{rem}
Let us now show the following unified improved $L^2$-Rellich type inequalities:
\begin{thm}\label{thm_Rellich_p2} Let $\mathbb{G}$ be a homogeneous Lie group
of homogeneous dimension $Q$. Let $|\cdot|$ be a homogeneous quasi-norm. Let $B(0,R)$ be a quasi-ball of $\mathbb{G}$ with radius $R$ with respect to $|\cdot|$.
\begin{itemize}
    \item If $4-Q<a \leq Q-c$ and $c>0$, then we have
\begin{equation}\label{(24)}
\left(\frac{Q+a-4}{4} c\right)^{2} \int_{B(0,R)} \frac{|f|^{2}}{|x|^{a}\left(1-\left(\frac{|x|}{R}\right)^{c}\right)^{2}} \leq \int_{B(0,R)} \frac{\left|\mathcal{R}_{|x|}^{2}f+\frac{Q-1}{|x|} \mathcal{R}_{|x|} f\right|^{2}}{|x|^{a-4}} d x    
\end{equation}
for all $f \in C_{0}^{\infty}\left(B(0,R) \backslash\{0\}\right)$. Moreover, when $c=Q-a$, the constant $\left(\frac{Q+a-4}{4} c\right)^{2}$ in \eqref{(24)} is optimal and is not attained for $f \not \equiv 0$ for which the right-hand side is finite.
\item If $3 \leq a \leq \min \{Q-c+2, Q-3 c\}$ and $c>0$, then we have
\begin{equation}\label{(25)}
\left(\frac{3}{4} c^{2}\right)^{2} \int_{B(0,R)} \frac{|f|^{2}}{|x|^{a}\left(1-\left(\frac{|x|}{R}\right)^{c}\right)^{4}} \leq \int_{B(0,R)} \frac{\left|\mathcal{R}_{|x|}^{2}f+\frac{Q-1}{|x|} \mathcal{R}_{|x|} f\right|^{2}}{|x|^{a-4}} d x
\end{equation}
for all $f \in C_{0}^{\infty}\left(B(0,R) \backslash\{0\}\right)$. Furthermore, the constant $\left(\frac{3}{4} c^{2}\right)^{2}$ in \eqref{(25)} is optimal and is not attained for $f \not \equiv 0$ for which the right-hand side is finite.
\end{itemize}
\end{thm}
\begin{rem} \begin{itemize} 
\item When $\G=(\Rn,+)$ and $Q=n\geq 5$, taking the Euclidean norm $|\cdot|_{E}$ we note that \eqref{(24)} implies the classical Rellich inequality when $a=4$ and $c=n-4$ preserving its optimal constant:
\begin{multline}\label{compare_rellich}
\int_{B(0,R)} \frac{|f|^{2}}{|x|^{4}} d x \leq \int_{B(0,R)} \frac{|f|^{2}}{|x|^{4}\left(1-\left(\frac{|x|}{R}\right)^{n-4}\right)^{2}} d x \\ \leq\left(\frac{n(n-4)}{4}\right)^{-2} \int_{B(0,R)}\left|\mathcal{R}_{|x|}^{2}f+\frac{n-1}{|x|} \mathcal{R}_{|x|} f\right|^{2} d x \\ \leq\left(\frac{n(n-4)}{4}\right)^{-2} \int_{B(0,R)}|\Delta f|^{2} d x
\end{multline}
for any $f \in C_{0}^{2}\left(B(0,R) \backslash\{0\}\right)$, where we have used \cite[Corollary 1.3]{MOW17} in the last line. While \eqref{(25)} implies the geometric Rellich inequality for any $f \in C_{0}^{2}\left(B(0,R) \backslash\{0\}\right)$ when $\G=(\Rn,+)$, $Q=n\geq 7$, $a=4$ and $c=1$: 
\begin{multline*}
\int_{B(0,R)} \frac{|f|^{2}}{(R-|x|)^{4}} d x \leq \int_{B(0,R)} \frac{|f|^{2}}{|x|^{4}\left(1-\frac{|x|}{R}\right)^{4}} d x \\ \leq\left(\frac{3}{4}\right)^{-2} \int_{B(0,R)}\left|\mathcal{R}_{|x|}^{2}f+\frac{n-1}{|x|} \mathcal{R}_{|x|} f\right|^{2} d x \leq\left(\frac{3}{4}\right)^{-2} \int_{B(0,R)}|\Delta f|^{2} d x.
\end{multline*}
\item In the Abelian case $\G=(\Rn,+)$ and $Q=n$ with the standard Euclidean norm $|\cdot|_{E}$, this theorem was recently obtained in \cite[Theorem 3.1]{San21} with the full Laplacian on the right-hand side. Here, Theorem \ref{thm_Rellich_p2} can be regarded as improved and generalised version of it since Theorem \ref{thm_Rellich_p2} even in the Eulidean case is obtained with the radial part of the Laplacian and with any homogeneous quasi-norm keeping their optimal constants. We also refer to \cite{MOW17} and \cite{BMO23} for the unweighted Rellich inequalities with the radial part of the Laplacian.
\item On homogeneous Lie groups, this theorem improves the Rellich inequalities from \cite{RS17} in the sense of \eqref{compare_rellich} with general $a, b, c$ as in Theorem \ref{thm_Rellich_p2}.
\item Moreover, when $c \rightarrow 0$ by the argument from Remark \ref{rem1} the inequalities \eqref{(24)} and \eqref{(25)} yield the following critical Rellich type inequalities, respectively, 
$$
\left(\frac{Q+a-4}{4}\right)^{2} \int_{B(0,R)} \frac{|f|^{2}}{|x|^{a}\left(\log \frac{R}{|x|}\right)^{2}} \leq \int_{B(0,R)} \frac{\left|\mathcal{R}_{|x|}^{2}f+\frac{Q-1}{|x|} \mathcal{R}_{|x|} f\right|^{2}}{|x|^{a-4}} d x
$$
and
$$
\left(\frac{3}{4}\right)^{2} \int_{B(0,R)} \frac{|f|^{2}}{|x|^{a}\left(\log \frac{R}{|x|}\right)^{4}} \leq \int_{B(0,R)} \frac{\left|\mathcal{R}_{|x|}^{2}f+\frac{Q-1}{|x|} \mathcal{R}_{|x|} f\right|^{2}}{|x|^{a-4}} d x.
$$ In the Abelian case $\G=(\Rn,+)$ and $Q=n$ with the usual Euclidean norm $|\cdot|_{E}$, these critical Rellich type inequalities give improved versions of the results in \cite{CM12}.
\end{itemize}
\end{rem}
Now we show unified $L^p$-Rellich type inequalities with remainder terms:
\begin{thm}\label{thm_Lp_Rel} Let $\mathbb{G}$ be a homogeneous Lie group
of homogeneous dimension $Q$. Let $|\cdot|$ be a homogeneous quasi-norm. Let $B(0,R)$ be a quasi-ball of $\mathbb{G}$ with radius $R$ with respect to $|\cdot|$.
Let $1<p, b<\infty$, $c>0$ and $a \leq Q-(p-1) c$. Then we have 
\begin{multline}\label{(35)}
\left(\frac{(p-1)c(Q(p-1)+a-2p)}{p^{2}} \right)^{p} \int_{B(0,R)} \frac{|f|^{p}}{|x|^{a}\left(1-\left(\frac{|x|}{R}\right)^{c}\right)^{p}}\\+\left(\frac{(Q(p-1)+a-2p)}{p}\right)^{p} \psi_{Q, p, a, p}(f) 
\leq \int_{B(0,R)} \frac{\left|\mathcal{R}_{|x|}^{2}f+\frac{Q-1}{|x|} \mathcal{R}_{|x|} f\right|^{p}}{|x|^{a-2p}} d x
    \end{multline}
for any $f \in C_{0}^{2}\left(B(0,R)\backslash\{0\}\right)$, where $\psi_{Q, p, a, b}(f)$ is given in Theorem \ref{T IH remainder}. Moreover, if we drop the second term on the left-hand side, then the constant in \eqref{(35)} when $c=\frac{Q-a}{p-1}$ is optimal and is not attained for $f \not \equiv 0$ for which the right-hand side is finite.
\end{thm}
\begin{rem}
\begin{itemize}
    \item In the Euclidean case $\G=(\Rn,+)$ and $Q=n$, this theorem gives an improved version of \cite[Theorem 3.6, Part I]{San21}.
    \item On homogeneous Lie groups, this theorem improves the Rellich inequalities from \cite{Ngu17} in the sense of \eqref{compare_rellich} with general $a, b, c$ as in Theorem \ref{thm_Lp_Rel}.
    \end{itemize}
\end{rem}

\section{Proofs}\label{S IH}

\begin{proof}[Proof of Theorem \ref{T IH remainder}]
For the simplicity, we give the proof of Theorem \ref{T IH remainder} when $R=1$. A direct calculation gives that
\begin{multline*}
(b -1)c \int_{B(0,1)} \frac{|f|^p}{|x|^a \( 1-|x|^c \)^b} \,dx 
= \int_{B(0,1)} \R_{|x|} \( \frac{1}{\( |x|^{-c} -1 \)^{b-1}} \) \frac{|f|^p}{|x|^{a + (b -1)c-1}}dx \\= -p {\rm Re}\int_{0}^{1}\int_{\wp}\frac{|f|^{p-2}f }{|x|^{a+c(b-1)-1 }\( |x|^{-c}-1 \)^{b-1}} \overline{\frac{df(ry)}{dr}}d\sigma(y)dr \\ - (Q-a -(b -1)c ) \int_{B(0,1)}\frac{|f|^p}{|x|^{a+c(b-1)} \( |x|^{-c}-1 \)^{b -1}} \,dx\\
 \leq p \int_{B(0,1)}\frac{|f|^{p-1}}{|x|^{a-1 }\(1- |x|^{c}\)^{b-1}} |\R_{|x|} f(x)|dx \\ - (Q-a -(b -1)c ) \int_{B(0,1)}\frac{|f|^p}{|x|^{a} \(1-|x|^{c}\)^{b -1}} dx\\=p \left(\int_{B(0,1)}\frac{|\R_{|x|} f|^{p}}{|x|^{a-p }\(1- |x|^{c}\)^{b-p}}dx\right)^{\frac{1}{p}}\left(\int_{B(0,1)}\frac{|f|^{p}}{|x|^{a}\(1- |x|^{c}\)^{b}}dx\right)^{1-\frac{1}{p}} \\ - (Q-a -(b -1)c ) \int_{B(0,1)}\frac{|f|^p}{|x|^{a} \(1-|x|^{c}\)^{b -1}},
\end{multline*}
that is, 
\begin{multline}\label{p_mae}
\left(\frac{b-1}{p} c\right)^{p}  \int_{B(0,1)} \frac{|f|^p}{|x|^a \( 1-|x|^c \)^b} \,dx  
\le \left( \( \int_{B(0,1)} \frac{\left| \R_{|x|} f \right|^p}{|x|^{a -p} \( 1-|x|^c \)^{b -p}} dx \)^{\frac{1}{p}} \right.\\\left.
-\frac{(Q-a -(b -1)c )}{p} \frac{\int_{B(0,1)} \frac{|f|^p}{|x|^a \( 1-|x|^c \)^{b -1}} dx}{\( \int_{B(0,1)} \frac{|f|^p}{|x|^a \( 1-|x|^c \)^b} dx  \)^{\frac{p-1}{p}}}\right)^{p}
\end{multline}
for any non-trivial $f \in C_0^1(B(0,1)\backslash\{0\})$. Here, taking into account the fundamental inequality $(\alpha-\beta)^p \le \alpha^p -p (\alpha-\beta)^{p-1} \beta$ for $\alpha \ge \beta$, we obtain
$$\(\frac{b-1}{p} c \)^p \int_{B(0,1)} \frac{|f|^p}{|x|^a \( 1-|x|^c \)^b} \,dx
\le \int_{B(0,1)} \frac{\left| \R_{|x|} f \right|^p}{|x|^{a -p} \( 1-|x|^c \)^{b -p}} dx $$
\begin{multline*}
- p\left( \( \int_{B(0,1)} \frac{\left| \R_{|x|} f \right|^p dx}{|x|^{a -p} \( 1-|x|^c \)^{b -p}}  \)^{\frac{1}{p}} 
-\frac{(Q-a -(b -1)c )}{p} \frac{\int_{B(0,1)} \frac{|f|^p dx}{|x|^a \( 1-|x|^c \)^{b -1}}}{\( \int_{B(0,1)} \frac{|f|^p dx}{|x|^a \( 1-|x|^c \)^b}  \)^{\frac{p-1}{p}}}\right)^{p-1}\\
\times\frac{(Q-a -(b -1)c )}{p} \frac{\int_{B(0,1)} \frac{|f|^p}{|x|^a \( 1-|x|^c \)^{b -1}} \,dx}{\( \int_{B(0,1)} \frac{|f|^p}{|x|^a \( 1-|x|^c \)^b} \,dx  \)^{\frac{p-1}{p}}},
\end{multline*}
which together with \eqref{p_mae} implies (\ref{IH remainder}) for $c<\frac{Q-a}{b-1}$. 

Now, let us consider the case $a = Q- (b -1)c$. Observe that
\begin{multline}
\int_{B(0,1)} \frac{\left| \R_{|x|} f \right|^p}{|x|^{a-p}\( 1-|x|^c \)^{b-p}} dx\\=
\int_{B(0,1)} \frac{\left| (|x|^{-c} -1)^{\frac{b -1}{p}} \( \R_{|x|} g\)- \frac{b -1}{p}c g \,|x|^{-c -1} (|x|^{-c} -1)^{\frac{b -1}{p} -1}\right|^p}{|x|^{a-p}\( 1-|x|^c \)^{b-p}} dx,
\end{multline}
where $g(x)= f(x) \( |x|^{-c} -1 \)^{-\frac{b -1}{p} }$. Then, using $|\alpha-\beta|^p - |\alpha|^p + p|\alpha|^{p-2} \alpha \beta \ge C_1 |\beta|^p$ with 
\begin{equation*}
\alpha= \frac{b -1}{p}c g \,|x|^{-c -1} (|x|^{-c} -1)^{\frac{b -1}{p} -1} \quad \text{and} \quad \beta= (|x|^{-c} -1)^{\frac{b -1}{p}} \( \R_{|x|} g\)
\end{equation*}
for $2\leq p <\infty$ and for some positive constant $C_{1}$ (see e.g. \cite[p. 3415]{FS08}), we have
\begin{multline}
\int_{B(0,1)} \frac{\left| \R_{|x|} f \right|^p}{|x|^{a-p}\( 1-|x|^c \)^{b-p}} dx\geq 
\int_{B(0,1)} \frac{|\alpha|^p -p |\alpha|^{p-2}\alpha \beta + C_1|\beta|^p}{|x|^{a-p}\( 1-|x|^c \)^{b-p}} dx\\=
\( \frac{b -1}{p} c \)^{p} \int_{B(0,1)} \frac{|g|^{p}}{|x|^Q \( 1- |x|^c \) } dx+
 \( \frac{b -1}{p} c \)^{p-1} \int_{B(0,1)} \frac{\R_{|x|} \( |g|^p \) }{|x|^{Q-1}} dx\\ + C_1 \int_{B(0,1)} |x|^{p-Q}\( 1- |x|^c \)^{p-1} \left| \R_{|x|} g  \right|^p dx\\=
 \( \frac{b -1}{p} c \)^{p} \int_{B(0,1)} \frac{|g|^{p}}{|x|^Q \( 1- |x|^c \) } dx + C_1 \int_{B(0,1)} |x|^{p-Q}\( 1- |x|^c \)^{p-1} \left| \R_{|x|} g  \right|^p dx,
\end{multline}
which implies (\ref{IH remainder}) for $2\leq p <\infty$ and $a = Q- (b -1)c$.

In the case $1<p<2$ similarly as above but using $|\alpha-\beta|^p - |\alpha|^p + p|\alpha|^{p-2} \alpha \beta \geq
C_2 \frac{|\beta|^2}{(|\alpha-\beta| + |\alpha|)^{2-p}}$ for some positive constant $C_{2}$ (see e.g. \cite[Lemma 4.2]{Lin90}) we get
\begin{multline}
\int_{B(0,1)} \frac{\left| \R_{|x|} f \right|^p}{|x|^{a-p}\( 1-|x|^c \)^{b-p}} dx\geq \( \frac{b -1}{p} c \)^{p} \int_{B(0,1)} \frac{|g|^{p}}{|x|^Q \( 1- |x|^c \) } dx \\+ 
 C_2 \int_{B(0,1)} \frac{|\beta|^2}{|x|^{a-p}\( 1-|x|^c \)^{b-p} (|\alpha-\beta| + |\alpha|)^{2-p}} \,dx \notag\\
=\( \frac{b -1}{p} c \)^{p} \int_{B(0,1)} \frac{|g|^{p}}{|x|^Q \( 1- |x|^c \) } dx \\+C_2 \int_{B(0,1)}\frac{\left(|x|^{(p-Q)\frac{2}{p}} (1-|x|^c)^{(p-1)\frac{2}{p}} \left| \R_{|x|} g  \right|^2 \right)|x|^{(p-a)\frac{p-2}{p}} (1-|x|^c)^{(p-b)\frac{p-2}{p}}}
{\( \left| \R_{|x|} f \right|
+
\frac{b-1}{p} c |g| |x|^{-1-\frac{Q-a}{p}} \( 1-|x|^c \)^{\frac{b-1}{p} -1}  \)^{2-p}} \,dx,
\end{multline}
where we have used $|\alpha-\beta|=|\R_{|x|}f|$ in the last equality. Using the reverse H\"older inequality for $\frac{2}{p}+\frac{p-2}{p}=1$ with $1<p<2$ we obtain
\begin{multline}
\int_{B(0,1)} \frac{\left| \R_{|x|} f \right|^p}{|x|^{a-p}\( 1-|x|^c \)^{b-p}} dx\geq \( \frac{b -1}{p} c \)^{p} \int_{B(0,1)} \frac{|g|^{p}}{|x|^Q \( 1- |x|^c \) } dx \\+C_2 
\left(\int_{B(0,1)} |x|^{p-Q}\( 1- |x|^c \)^{p-1} \left| \R_{|x|} g\right|^p \,dx\right)^{\frac{2}{p}}\\ \times
\( \int_{B(0,1)} \left| \, \left| \R_{|x|} f  \right|
+
\frac{b-1}{p} c \frac{|g|}{|x|^{1+\frac{Q-a}{p}}} \( 1-|x|^c \)^{\frac{b-1}{p} -1} \right|^p |x|^{p-a} (1-|x|^c)^{p-b}\,dx \)^{\frac{p-2}{p}}
\end{multline}
yielding (\ref{IH remainder}) for $p \in (1, 2)$ and $a = Q- (b -1)c$. 

%%%%%%%%%%%%%%%%%%%%%%%%%%%%%%%%%%

To complete proof of Theorem \ref{T IH remainder}, we need to show the optimality of the constant $\( \frac{b-1}{p} c \)^p$ in the inequality \eqref{IH remainder} after dropping the non-negative term $\psi_{Q,p,a,b}(f)$.

%%%%%%%%%%%%%%%%%%%%%%%%%%%%%%%%%%
As in \cite[Theorem 1.1]{San21}, for $\kappa > \frac{b-1}{p}$ and small $\delta >0$ one can construct
\begin{align}\label{test_func}
	u_{\kappa} (x)= \phi_\delta (x) \,(1-|x|^c)^\kappa,
\end{align}
where $\phi_\delta=\phi_\delta(|x|)$ is a smooth function with $0 \le \phi_\delta \le 1, \phi_\delta \equiv 0$ when $|x|<1-2\delta$ and $\phi_\delta \equiv 1$ when $1-\delta<|x|<1$, so that
\begin{multline*}
\( \frac{b-1}{p} c \)^p 
\le \frac{\int_{B(0,1)} \frac{\left| \R_{|x|} u_{\kappa} \right|^p}{|x|^{a -p} \( 1-|x|^c \)^{b -p}} \,dx}{\int_{B(0,1)} \frac{|u_{\kappa}|^p}{|x|^a \( 1-|x|^c \)^b} \,dx} \\
\le \frac{|\wp|\int_{1-2\delta}^{1-\delta} |\R_{r}u_{\kappa}|^p \( 1-r^c \)^{p-b} r^{Q-1-a+p}\,dr+|\wp|(\kappa c)^p \int_{1-\delta}^1 \( 1-r^c \)^{\kappa p-b} r^{Q-1-a + c p}\,dr}{|\wp|\int_{1-\delta}^{1} \( 1-r^c \)^{\kappa p-b} r^{Q-1-a}\,dr} \\
= \( \frac{b-1}{p}c \)^p + o(1),
\end{multline*}
as $\kappa \to \frac{b-1}{p}$, which shows the optimality of the constant $\left(\frac{b-1}{p} c\right)^p$. Here $|\wp|$ is $Q-1$ dimensional surface measure of the unit quasi-sphere with respect to $|\cdot|$. Existence of the non-negative remainder term in \eqref{IH remainder} when $a \le Q - (b -1)c$ shows that if there exists an extremal function $v=v(x)$ of the inequality \eqref{IH remainder} without the non-negative term $\psi_{Q,p,a,b}(f)$, then $v(x)= C ( |x|^{-c} -1 )^{\frac{b -1}{p}} = C |x|^{-\frac{b -1}{p}c} (  1-|x|^c )^{\frac{b -1}{p}}$ for some $C \in \re$. Note that if $C \not= 0$, then we have 
\begin{multline*}
\int_{B(0,1) \setminus B(0,1-\ep)} \frac{\left| \R_{|x|} v\right|^p}{|x|^{a -p} \( 1-|x|^c \)^{b -p}} dx
\\ \ge C(\ep)|\wp| \int_{1-\ep}^{1} \frac{\left| \R_{r} \(  1-r^c \)^{\frac{b -1}{p}} \right|^p}{ \( 1-r^c \)^{b -p}} r^{Q-1}dr + D(\ep)\\
\ge \widetilde{C}(\ep) |\wp|\int_{1-\ep}^{1} \frac{r^{Q-1}}{1-r^c} \,dr + D(\ep) =\infty
\end{multline*}
for any small $\ep >0$, that is, the right-hand side of the inequality \eqref{IH remainder} without the non-negative term $\psi_{Q,p,a,b}(f)$ diverges, where $C(\ep)$, $\widetilde{C}(\ep)$ and $D(\ep)$ are some constants.
\end{proof}

\begin{proof}[Proof of Theorem \ref{equality_Hardy}] As above, for convenience let us take $R=1$. By a direct calculation, we get
\begin{multline}\label{lp_iden_th1.5}
\int_{B(0,1)} \frac{|f|^{p}}{|x|^{a}\left(1-|x|^{c}\right)^{b}} d x 
=\int_{0}^{1} r^{Q-a-1}\left(1-r^{c}\right)^{-b} \int_{\wp}|f(r y)|^{p} d\sigma(y) d r  \\
=\frac{1}{(b-1) c} \int_{0}^{1} \int_{\wp}|f|^{p} r^{Q-a-(b-1) c} \frac{d}{d r}\left[\left(r^{-c}-1\right)^{-b+1}\right] d\sigma(y) d r  \\
=\frac{p}{(b-1) c} \int_{0}^{1} \int_{\wp}|f|^{p-2} f(-\R_{r} f) r^{Q-a}\left(1-r^{c}\right)^{-b+1} d\sigma(y) d r  \\
-\frac{Q-a-(b-1) c}{(b-1) c} \int_{0}^{1} \int_{\wp}|f|^{p} r^{Q-a-1}\left(1-r^{c}\right)^{-b+1} d\sigma(y) d r\\=\frac{p}{(b-1) c} \int_{B(0,1)}|f|^{p-2} f(-\R_{|x|} f) |x|^{1-a}\left(1-|x|^{c}\right)^{-b+1} dx  \\
-\frac{Q-a-(b-1) c}{(b-1) c} \int_{B(0,1)}\frac{|f|^{p}}{ |x|^{a}\left(1-|x|^{c}\right)^{b-1}} dx.\end{multline}
Setting
$$
\widetilde{u}=\frac{p}{(b-1) c}\frac{-\R_{|x|} f}{|x|^{\frac{a-p}{p}}\left(1-|x|^{c}\right)^{\frac{b-p}{p}}}\quad \text{and} \quad \widetilde{v}=\frac{f}{|x|^{\frac{a}{p}}(1-|x|^{c})^{\frac{b}{p}}},$$
we rewrite \eqref{lp_iden_th1.5} as
\begin{equation}\label{lp_iden_th1.5_2}
\|\widetilde{v}\|^{p}_{B(0,1)} 
=\int_{B(0,1)}|\widetilde{v}|^{p-2} \widetilde{v} \widetilde{u} dx -\frac{Q-a-(b-1) c}{(b-1) c} \int_{B(0,1)} \frac{|f|^{p}}{|x|^{a}\left(1-|x|^{c}\right)^{b-1}} d x.
\end{equation}
On the other hand, for any $L^p$-integrable real-valued functions $\widetilde{u}$ and $\widetilde{v}$, one has
\begin{multline}\label{lp_iden_argum_for}
    \left\|\widetilde{u}\right\|^p_{L^p(B(0,1))}-\left\|\widetilde{v}\right\|^p_{L^p(B(0,1))}+p\int_{B(0,1)}(|\widetilde{v}|^p-|\widetilde{v}|^{p-2}\widetilde{v}\widetilde{u})dx
    \\=\int_{B(0,1)}(|\widetilde{u}|^p+(p-1)|\widetilde{v}|^{p}-p|\widetilde{v}|^{p-2}\widetilde{v}\widetilde{u})dx
    \\=p\int_{B(0,1)}\frac{\frac{|\widetilde{u}|^{p}}{p}+\frac{p-1}{p}|\widetilde{v}|^{p}-|\widetilde{v}|^{p-2}\widetilde{v}\widetilde{u}}{|\widetilde{v}-\widetilde{u}|^{2}}|\widetilde{v}-\widetilde{u}|^{2}dx=p\int_{B(0,1)}I_{p}(\widetilde{v},\widetilde{u})|\widetilde{v}-\widetilde{u}|^{2}dx,
\end{multline}
where we have used \cite[Proposition 1.1]{IIO17} in the last line. Combining this with \eqref{lp_iden_th1.5_2} we obtain
\begin{multline*}\left\|\widetilde{u}\right\|^p_{L^p(B(0,1))}-\left\|\widetilde{v}\right\|^p_{L^p(B(0,1))}-p\frac{Q-a-(b-1) c}{(b-1) c} \int_{B(0,1)} \frac{|f|^{p}}{|x|^{a}\left(1-|x|^{c}\right)^{b-1}} d x\\=p\int_{B(0,1)}I_{p}(\widetilde{v},\widetilde{u})|\widetilde{v}-\widetilde{u}|^{2}dx,\end{multline*}
which implies \eqref{eq_Lp_iden} with $R=1$ after substituting
$$\widetilde{u}=\frac{u}{|x|^{\frac{a}{p}}(1-|x|^{c})^{\frac{b}{p}}}\quad\text{and}\quad \widetilde{v}=\frac{f}{|x|^{\frac{a}{p}}(1-|x|^{c})^{\frac{b}{p}}}.$$
\end{proof}
\begin{proof}[Proof of Theorem \ref{thm_high_1.13}] We can iterate \eqref{eq_L2_iden}, that is, for any $b>1$, $a<Q$ and $0<c\leq \frac{Q-a}{b-1}$ we have  
\begin{multline}\label{eq_L2_iden_pr1}
 \left(\frac{b-1}{2} c\right)^2 \int_{B(0,R)} \frac{|f|^2}{|x|^a\left(1-\left(\frac{|x|}{R}\right)^c\right)^b} d x=\int_{B(0,R)} \frac{\left|\mathcal{R}_{|x|} f\right|^2}{|x|^{a-2}\left(1-\left(\frac{|x|}{R}\right)^c\right)^{b-2}} d x \\
 -\frac{((b-1) c)^2}{2} \int_{B(0,R)}\left|f+\frac{2}{(b-1) c}\left(\mathcal{R}_{|x|} f\right)|x|\left(1-\left(\frac{|x|}{R}\right)^c\right)\right|^2 
 \frac{d x}{|x|^a\left(1-\left(\frac{|x|}{R}\right)^c\right)^b} \\
 -(Q-a-(b-1) c)\left(\frac{b-1}{2} c\right) \int_{B(0,R)} \frac{|f|^2}{|x|^a\left(1-\left(\frac{|x|}{R}\right)^c\right)^{b-1}} d x. 
\end{multline}
In \eqref{eq_L2_iden_pr1} replacing $f$ by $\mathcal{R}_{|x|} f$ we get
\begin{multline*}
 \left(\frac{b-1}{2} c\right)^2 \int_{B(0,R)} \frac{|\mathcal{R}_{|x|} f|^2}{|x|^a\left(1-\left(\frac{|x|}{R}\right)^c\right)^b} d x=\int_{B(0,R)} \frac{\left|\mathcal{R}_{|x|}^{2} f\right|^2}{|x|^{a-2}\left(1-\left(\frac{|x|}{R}\right)^c\right)^{b-2}} d x \\
 -\frac{((b-1) c)^2}{2} \int_{B(0,R)}\left|\mathcal{R}_{|x|} f+\frac{2}{(b-1) c}\left(\mathcal{R}_{|x|}^{2} f\right)|x|\left(1-\left(\frac{|x|}{R}\right)^c\right)\right|^2 
 \frac{d x}{|x|^a\left(1-\left(\frac{|x|}{R}\right)^c\right)^b}
 \end{multline*}
 \begin{equation}\label{eq_L2_iden_pr2}
 -(Q-a-(b-1) c)\left(\frac{b-1}{2} c\right) \int_{B(0,R)} \frac{|\mathcal{R}_{|x|} f|^2}{|x|^a\left(1-\left(\frac{|x|}{R}\right)^c\right)^{b-1}} d x. 
\end{equation}
On the other hand, replacing $a$ by $a+2$ and $b$ by $b+2$ in \eqref{eq_L2_iden_pr1} gives
\begin{multline*}
 \left(\frac{b+1}{2} c\right)^2 \int_{B(0,R)} \frac{|f|^2}{|x|^{a+2}\left(1-\left(\frac{|x|}{R}\right)^c\right)^{b+2}} d x=\int_{B(0,R)} \frac{\left|\mathcal{R}_{|x|} f\right|^2}{|x|^{a}\left(1-\left(\frac{|x|}{R}\right)^c\right)^{b}} d x \\
 -\frac{((b+1) c)^2}{2} \int_{B(0,R)}\left|f+\frac{2}{(b+1) c}\left(\mathcal{R}_{|x|} f\right)|x|\left(1-\left(\frac{|x|}{R}\right)^c\right)\right|^2 
 \frac{d x}{|x|^{a+2}\left(1-\left(\frac{|x|}{R}\right)^c\right)^{b+2}} 
\end{multline*}
\begin{equation}\label{eq_L2_iden_pr1_shift}
 -(Q-a-2-(b+1) c)\left(\frac{b+1}{2} c\right) \int_{B(0,R)} \frac{|f|^2}{|x|^{a+2}\left(1-\left(\frac{|x|}{R}\right)^c\right)^{b+1}} d x,
 \end{equation}
where $b+2>1, a+2<Q$ and $0<c\leq \frac{Q-a-2}{b+1}$.
Combining this with \eqref{eq_L2_iden_pr2} we obtain
\begin{multline}\label{eq_high_L2_pr3} 
\prod_{j=0}^{1}\left(\frac{b+2 j-1}{2} c\right)^2 \int_{B(0,R)} \frac{|f|^2}{|x|^{a+2}\left(1-\left(\frac{|x|}{R}\right)^c\right)^{b+2}} d x 
+2 \sum_{i=1}^2 \prod_{j=0}^{i-1}\left(\frac{b+2 j-1}{2} c\right)^2 \\
\int_{B(0,R)}\left|\mathcal{R}_{|x|}^{2-i} f+\frac{2}{(b+2(i-1)-1) c}\left(\mathcal{R}_{|x|}^{3-i} f\right)|x|\left(1-\left(\frac{|x|}{R}\right)^c\right)\right|^2 
\\ \times \frac{d x}{|x|^{a+2(i-1)}\left(1-\left(\frac{|x|}{R}\right)^c\right)^{b+2(i-1)}} \\
 +\sum_{i=1}^2 \prod_{j=0}^{i-1}(Q-a-2(i-1)-(b+2(i-1)-1) c)\left(\frac{b+2 j-1}{2} c\right)^2\left(\frac{b+2(i-1)-1}{2} c\right)^{-1} \\
 \times\int_{B(0,R)} \frac{\left|\mathcal{R}_{|x|}^{2-i} f\right|^2}{|x|^{a+2(i-1)}\left(1-\left(\frac{|x|}{R}\right)^c\right)^{b+2(i-1)-1} }d x 
 =\int_{B(0,R)} \frac{\left|\mathcal{R}_{|x|}^2 f\right|^2}{|x|^{a-2}\left(1-\left(\frac{|x|}{R}\right)^c\right)^{b-2} }d x.
\end{multline}
Repeating this iteration process gives \eqref{eq_high_L2}.
\end{proof}
\begin{proof}[Proof of Theorem \ref{thm_high_1.14}] As in the proof of Theorem \ref{thm_high_1.13} above, by iterating \eqref{eq_Lp_iden} and \eqref{IH remainder} without the remainder term  we obtain \eqref{eq_high_Lp} and \eqref{eq_high_Lp_ineq}, respectively.
\end{proof}
\begin{proof}[Proof of Theorem \ref{thm_CKN_unif2}] {\bf Case $\delta=0$}. In this case, we have $q=r$ and $\beta=\gamma$ by $\frac{\delta r}{p}+\frac{(1-\delta)r}{q}=1$ and $\gamma=-\delta+\beta(1-\delta)$, respectively. Then, inequality \eqref{CKN_thm2} becomes
$$
\|\omega^{\beta}f\|_{L^{q}(B(0,R))}
\leq \left\|\omega^{\beta}f\right\|_{L^{q}(B(0,R))},
$$
which is trivial.

{\bf Case $\delta=1$}. In this case, we have $p=r$ and $\gamma=-1$. So, the inequality \eqref{CKN_thm2} becomes \eqref{eq_high_Lp}.

{\bf Case $\delta\in(0,1)\cap\left[\frac{r-q}{r},\frac{p}{r}\right]$}. 
Taking into account $\gamma=-\delta+\beta(1-\delta)$,  a direct calculation gives \begin{multline*}
    \|\omega^{\gamma}f\|_{L^{r}(B(0,R))}=
\left(\int_{B(0,R)}(\omega(x))^{\gamma r}|f(x)|^{r}dx\right)^{\frac{1}{r}}\\
=\left(\int_{B(0,R)}\frac{|f(x)|^{\delta r}}{(\omega(x))^{\delta r }}\cdot \frac{|f(x)|^{(1-\delta)r}}{(\omega(x))^{-\beta r(1-\delta)}}dx\right)^{\frac{1}{r}}.
\end{multline*}
Since we have $\delta\in(0,1)\cap\left[\frac{r-q}{r},\frac{p}{r}\right]$ and $p+q\geq r$, then by using H\"{o}lder's inequality for $\frac{\delta r}{p}+\frac{(1-\delta)r}{q}=1$, we obtain
$$\|\omega^{\gamma}f\|_{L^{r}(B(0,R))}
\leq \left(\int_{B(0,R)}\frac{|f(x)|^{p}}{\omega(x)^{p}}dx\right)^{\frac{\delta}{p}}
\left(\int_{B(0,R)}\frac{|f(x)|^{q}}{\omega(x)^{-\beta q}}dx\right)^{\frac{1-\delta}{q}}$$
\begin{equation}\label{CKN_thm1_1}=\left\|\frac{f}{\omega}\right\|^{\delta}_{L^{p}(B(0,R))}
\left\|\frac{f}{\omega^{-\beta}}\right\|^{1-\delta}_{L^{q}(B(0,R))}.
\end{equation}
By \eqref{eq_high_Lp}, we have
\begin{equation}\label{1}
\left\|\frac{f}{\omega}\right\|^{\delta}_{L^{p}(B(0,R))}\leq
 \prod_{j=0}^{k-1}\left(\frac{p}{c(b+j p-1)} \right)^\delta \left(\left\|\frac{\mathcal{R}_{|x|}^k f}{\|\left. x\right|^{\frac{a-p}{p}}\left(1-\left(\frac{|x|}{R}\right)^c\right)^{\frac{b-p}{p}}}\right\|_{L^p(B(0,R))}^p-\Re\right)^{\frac{\delta}{p}}.
\end{equation}
A combination of this and \eqref{CKN_thm1_1} yields \eqref{CKN_thm2} for all real-valued functions $f \in C_0^{k}\left(B(0,R)\backslash\{0\}\right)$. 

Similarly, using the inequalities \eqref{eq_high_L2} and \eqref{eq_high_Lp_ineq} instead of \eqref{eq_high_Lp} one can obtain \eqref{CKN_thm2} for any complex-valued function $f \in C_0^{k}\left(B(0,R)\backslash\{0\}\right)$ for the cases $p=2$ and for general $p$ without the remainder term, respectively.
\end{proof}
\begin{proof}[Proof of Theorem \ref{thm_gen_2ineq}] For convenience, set $R=1$. Note that
\begin{multline}\label{(21)}
(1-a)\int_{B(0,1)} \frac{|f|^{p}}{|x|^{a}\left(1-|x|^{c}\right)^{b-1}}dx+(b-1)c\int_{B(0,1)} \frac{|f|^{p}}{|x|^{a-c}\left(1-|x|^{c}\right)^{b}} d x\\=\int_{B(0,1)} \R_{|x|}\left(\frac{1}{|x|^{a-1}\left(1-|x|^{c}\right)^{b-1}}\right)|f|^{p} d x \\
=-p{\rm Re}\int_{B(0,1)}\frac{|f|^{p-2}f\overline{\mathcal{R}_{|x|}f}}{|x|^{a-1}\left(1-|x|^{c}\right)^{b-1}}dx-(Q-1)\int_{B(0,1)} \frac{|f|^{p}}{|x|^{a}\left(1-|x|^{c}\right)^{b-1}}dx,
\end{multline}
where we have used integration by parts in the last line. Dropping the second term on the left-hand side of \eqref{(21)} and using H\"older's inequality, we obtain
$$
\left(\frac{Q-a}{p}\right) \int_{B(0,1)} \frac{|f|^{p}}{|x|^{a}\left(1-|x|^{c}\right)^{b-1}} d x \leq \int_{B(0,1)} \frac{|f|^{p-1}\left|\R_{|x|} f\right|}{|x|^{a-1}\left(1-|x|^{c}\right)^{b-1}} d x 
$$
$$
\leq\left(\int_{B(0,1)} \frac{|f|^{p}}{|x|^{a}\left(1-|x|^{c}\right)^{b-1}} d x\right)^{1-\frac{1}{p}}\left(\int_{B(0,1)} \frac{\left|\R_{|x|} f\right|^{p}}{|x|^{a-p}\left(1-|x|^{c}\right)^{b-1}} d x\right)^{\frac{1}{p}},
$$
which implies the desired inequality \eqref{(7)} for $f \in C_{0}^{1}\left(B(0,1)\backslash\{0\}\right)$. 

The optimality of the constant $\left(\frac{Q-a}{p}\right)^{p}$ in \eqref{(7)} can be shown as in \cite[Theorem 1.4]{San21} if we consider the test function $u_{\kappa}(x)=\phi_\delta(x)|x|^{\kappa}$ for $\kappa<\frac{Q-a}{p}$ and small $\delta>0$, where $\phi_\delta=\phi_\delta(|x|)$ is a smooth function with $0 \leq \phi_\delta \leq 1, \phi_\delta \equiv 0$ when $2\delta<|x|<1$ and $\phi_\delta \equiv 1$ when $|x|<\delta$. The non-attainability of the constant when $b=1$ is known due to \cite[Theorem 3.1]{RS17} since in this case \eqref{(7)} becomes the Hardy inequality on $\G$, whereas if $b>1$ the constant is not attained for non-trivial $f$ since we have dropped the second term on the left-hand side of \eqref{(21)} above. 

Let us now prove Part (ii). Using notations
$$
u=-\frac{p}{Q-a} \frac{\mathcal{R}_{|x|} f}{|x|^{\frac{a-p}{p}}(1-|x|^{c})^{\frac{b-1}{p}}} \quad \text{and} \quad
v=\frac{f}{|x|^{\frac{a}{p}}(1-|x|^{c})^{\frac{b-1}{p}}},
$$
the identity \eqref{(21)} can be restated as
\begin{equation}\label{(3.8)}
\|v\|_{L^p(B(0,1))}^p+\frac{(b-1)c}{Q-a}\int_{B(0,1)} \frac{|f|^{p}}{|x|^{a-c}\left(1-|x|^{c}\right)^{b}} d x=\operatorname{Re} \int_{B(0,1)}|v|^{p-2} v \overline{u} d x .
\end{equation}
In the case of a real-valued $f$ formula \eqref{(21)} becomes
\begin{equation}\label{(3.9)}
\|v\|_{L^p(B(0,1))}^p+\frac{(b-1)c}{Q-a}\int_{B(0,1)} \frac{|f|^{p}}{|x|^{a-c}\left(1-|x|^{c}\right)^{b}} d x=\int_{B(0,1)}|v|^{p-2} v u d x .
\end{equation}
On the other hand, as in \eqref{lp_iden_argum_for} for any $L^p$-integrable real-valued functions $u$ and $v$, we have
$$
    \left\|u\right\|^p_{L^p(B(0,1))}-\left\|v\right\|^p_{L^p(B(0,1))}+p\int_{B(0,1)}(|v|^p-|v|^{p-2}vu)dx
    =p\int_{B(0,1)}I_{p}(v,u)|v-u|^{2}dx.
$$
Combining this with \eqref{(3.9)} we obtain
\begin{multline*}
\|u\|_{L^p(B(0,1))}^p-\|v\|_{L^p(B(0,1))}^p-\frac{(b-1)cp}{Q-a}\int_{B(0,1)} \frac{|f|^{p}}{|x|^{a-c}\left(1-|x|^{c}\right)^{b}} d x\\=p \int_{B(0,1)} I_p(v, u)|v-u|^2 d x .
\end{multline*}
The equality \eqref{(3.2)} is proved.

Let us now prove Part (iii). If $p=2$, the identity \eqref{(3.8)} takes the form
\begin{equation}\label{L2_iden_proof}
\|v\|_{L^2(B(0,1))}^2+\frac{(b-1)c}{Q-a}\int_{B(0,1)} \frac{|f|^{2}}{|x|^{a-c}\left(1-|x|^{c}\right)^{b}} d x=\operatorname{Re} \int_{B(0,1)} v \overline{u} d x .
\end{equation}
On the other hand, we have
$$
\|u\|_{L^2(B(0,1))}^2-\|v\|_{L^2(B(0,1))}^2+2 \int_{B(0,1)}\left(|v|^2-\operatorname{Re} v \overline{u}\right) d x =\int_{B(0,1)}|u-v|^2 d x,
$$
which together with \eqref{L2_iden_proof} gives \eqref{(3.3)}.
\end{proof}

\begin{proof}[Proof of Theorem \ref{thm_gen_2ineq_8}] As above, for convenience let us take $R=1$. We drop the first term on the left-hand side of \eqref{(21)} and then use the H\"older inequality to get the desired inequality \eqref{(8)}: 
$$
\begin{aligned}
&\left(\frac{b-1}{p} c\right) \int_{B(0,1)} \frac{|f|^{p}}{|x|^{a-c}\left(1-|x|^{c}\right)^{b}} d x \leq \int_{B(0,1)} \frac{|f|^{p-1}\left|\R_{|x|} f\right|}{|x|^{a-1}\left(1-|x|^{c}\right)^{b-1}} d x \\
&\leq\left(\int_{B(0,1)} \frac{|f|^{p}}{|x|^{a-c}\left(1-|x|^{c}\right)^{b}} d x\right)^{1-\frac{1}{p}}\left(\int_{B(0,1)} \frac{\left|\R_{|x|}f \right|^{p}}{|x|^{a-c+(c-1) p}\left(1-|x|^{c}\right)^{b-p}} d x\right)^{\frac{1}{p}}.
\end{aligned}
$$

As in the proof of Theorem \ref{T IH remainder}, by taking the same test function from \eqref{test_func} one can show optimality of the constant $\left(\frac{b-1}{p} c\right)^p$ in \eqref{(8)}. Also, one can observe that the optimal constant is not attained since we have dropped the first term on the left-hand side of \eqref{(21)} above for $a<Q$.

Let us now prove Part (ii). Using notations
$$
u=-\frac{p}{(b-1)c} \frac{\mathcal{R}_{|x|} f}{|x|^{\frac{a-p+c(p-1)}{p}}(1-|x|^{c})^{\frac{b-p}{p}}}\quad \text{and}\quad
v=\frac{f}{|x|^{\frac{a-c}{p}}(1-|x|^{c})^{\frac{b}{p}}},
$$
the formula \eqref{(21)} can be rewritten as
\begin{equation}\label{(3.8_8)}
\|v\|_{L^p(B(0,1))}^p+\frac{Q-a}{(b-1)c}\int_{B(0,1)} \frac{|f|^{p}}{|x|^{a}\left(1-|x|^{c}\right)^{b-1}} d x=\operatorname{Re} \int_{B(0,1)}|v|^{p-2} v \overline{u} d x .
\end{equation}
In the case of a real-valued $f$ formula \eqref{(21)} becomes
\begin{equation}\label{(3.9_8)}
\|v\|_{L^p(B(0,1))}^p+\frac{Q-a}{(b-1)c}\int_{B(0,1)} \frac{|f|^{p}}{|x|^{a}\left(1-|x|^{c}\right)^{b-1}} d x=\int_{B(0,1)}|v|^{p-2} v u d x .
\end{equation}
On the other hand, as in \eqref{lp_iden_argum_for} for any $L^p$-integrable real-valued functions $u$ and $v$, one has
$$
    \left\|u\right\|^p_{L^p(B(0,1))}-\left\|v\right\|^p_{L^p(B(0,1))}+p\int_{B(0,1)}(|v|^p-|v|^{p-2}vu)dx
    =p\int_{B(0,1)}I_{p}(v,u)|v-u|^{2}dx.
$$
A combination of this and \eqref{(3.9_8)} implies
\begin{multline*}
\|u\|_{L^p(B(0,1))}^p-\|v\|_{L^p(B(0,1))}^p-\frac{(Q-a)p}{(b-1)c}\int_{B(0,1)} \frac{|f|^{p}}{|x|^{a}\left(1-|x|^{c}\right)^{b-1}} d x\\=p \int_{B(0,1)} I_p(v, u)|v-u|^2 d x,
\end{multline*}
which is \eqref{(3.2_8)} as desired.

Let us now prove Part (iii). If $p=2$, the identity \eqref{(3.8_8)} takes the form
\begin{equation}\label{L2_iden_proof_8}
\|v\|_{L^2(B(0,1))}^2+\frac{Q-a}{(b-1)c}\int_{B(0,1)} \frac{|f|^{2}}{|x|^{a}\left(1-|x|^{c}\right)^{b-1}} d x=\operatorname{Re} \int_{B(0,1)} v \overline{u} d x .
\end{equation}
On the other hand, we have
$$
\|u\|_{L^2(B(0,1))}^2-\|v\|_{L^2(B(0,1))}^2+2 \int_{B(0,1)}\left(|v|^2-\operatorname{Re} v \overline{u}\right) d x =\int_{B(0,1)}|u-v|^2 d x,
$$
which together with \eqref{L2_iden_proof_8} gives \eqref{(3.3_8)}.
\end{proof}
To prove Theorem \ref{thm_Rellich_p2} we need the following proposition: 
\begin{prop}\label{prop_2}
Let $1\leq p<\infty$. Then we have 
\begin{equation}
    \label{(38)}
\left|\frac{Q(p-1)+a-p}{p}\right|^{p} \int_{B(0,R)} \frac{|\R_{|x|} f|^{p}}{|x|^{a}} d x \leq \int_{B(0,R)} \frac{\left|\mathcal{R}_{|x|}^{2}f+\frac{Q-1}{|x|} \mathcal{R}_{|x|} f\right|^{p}}{|x|^{a-p}} d x
\end{equation}
for any $f \in C_{0}^{2}\left(B(0,R)\backslash\{0\}\right)$.
\end{prop}
\begin{proof}[Proof of Proposition \ref{prop_2}] By a direct calculation and H\"older's inequality we obtain
\begin{multline*}\left(
\int_{B(0,R)} \frac{\left|\mathcal{R}_{|x|}^{2}f+\frac{Q-1}{|x|} \mathcal{R}_{|x|} f\right|^{p}}{|x|^{a-p}} d x\right)^{\frac{1}{p}} \left(
\int_{B(0,R)}  \frac{|\R_{|x|} f|^{p}}{|x|^{a}} dx\right)^{\frac{p-1}{p}} \\=\left(\int_{B(0,R)} |x|^{p(2-Q)-a}\left|\R_{|x|} (|x|^{Q-1}\R_{|x|}f)\right|^{p} dx\right)^{\frac{1}{p}}\left(
\int_{B(0,R)}  \frac{|\R_{|x|} f|^{p}}{|x|^{a}} dx\right)^{\frac{p-1}{p}}  \\
 \geq \int_{B(0,R)} \left(|x|^{\frac{p(2-Q)-a}{p}}|\R_{|x|} (|x|^{Q-1}\R_{|x|}f)|\right) \left(|x|^{-a\frac{p-1}{p}}|\R_{|x|} f|^{p-1}\right)dx\\\geq\left|-{\rm Re}\int_{\wp}\int_{0}^{R} r^{Q(1-p)-a+p}|r^{Q-1}\R_{r}f|^{p-2} (r^{Q-1}\R_{r}f) \overline{\R_{r} (r^{Q-1}\R_{r}f)} dr d\sigma(y)\right|\\=\left|\frac{Q(p-1)+a-p}{p}\right| \int_{B(0,R)}|x|^{-Qp-a+p}||x|^{Q-1}\R_{|x|}f|^{p}dx\\
=\left|\frac{Q(p-1)+a-p}{p}\right| \int_{B(0,R)} \frac{|\R_{|x|} f|^{p}}{|x|^{a}} d x
\end{multline*}
yielding \eqref{(38)}.
\end{proof}
Now, we are ready to prove Theorem \ref{thm_Rellich_p2}:
\begin{proof}[Proof of Theorem \ref{thm_Rellich_p2}] For convenience let us prove for $R=1$. Then for $f \in C_{0}^{\infty}\left(B(0,1)\backslash\{0\}\right)$ we have  
\begin{multline}\label{eq_1.14_1}
\int_{B(0,1)} \frac{\left|\mathcal{R}_{|x|}^{2}f+\frac{Q-1}{|x|} \mathcal{R}_{|x|} f\right|^{2}}{|x|^{a-4}} d x=\int_{0}^{1} \int_{\wp}\left|\mathcal{R}_{r}^{2}f+\frac{Q-1}{r}\mathcal{R}_{r}f\right|^{2} r^{Q-a+3} d\sigma(y) d r\\
=\int_{0}^{1} \int_{\wp}\left|\mathcal{R}_{r}^{2}f\right|^{2} r^{Q-a+3}d\sigma(y) d r+2(Q-1){\rm Re}\int_{0}^{1} \int_{\wp}\mathcal{R}_{r}^{2}f\overline{\mathcal{R}_{r}f}r^{Q-a+2}d\sigma(y) d r\\+(Q-1)^{2}\int_{0}^{1} \int_{\wp}\left|\mathcal{R}_{r}f\right|^{2} r^{Q-a+1} d\sigma(y) d r=\int_{0}^{1} \int_{\wp}\left|\mathcal{R}_{r}^{2}f\right|^{2} r^{Q-a+3}d\sigma(y) d r\\+(-(Q-1)(Q-a+2)+(Q-1)^{2})\int_{0}^{1} \int_{\wp}\left|\mathcal{R}_{r}f\right|^{2} r^{Q-a+1} d\sigma(y) d r\\=\int_{0}^{1} \int_{\wp}\left|\mathcal{R}_{r}^{2}f\right|^{2} r^{Q-a+3}d\sigma(y) d r+(Q-1)(a-3)\int_{0}^{1} \int_{\wp}\left|\mathcal{R}_{r}f\right|^{2} r^{Q-a+1} d\sigma(y) d r.
\end{multline}
If $4-Q<a$ and $Q-a-c \geq 0$, then using the Hardy inequality \eqref{(7)} with $p=2$ and $b=1$ for the first term on the right-hand side of the last identity, we get
\begin{multline*}
\int_{B(0,1)} \frac{\left|\mathcal{R}_{|x|}^{2}f+\frac{Q-1}{|x|} \mathcal{R}_{|x|} f\right|^{2}}{|x|^{a-4}} d x \\ \geq \left(\left(\frac{Q-a+2}{2}\right)^{2}+(Q-1)(a-3)\right)\int_{0}^{1} \int_{\wp}\left|\mathcal{R}_{r}f\right|^{2} r^{Q-a+1}d\sigma(y) d r
\end{multline*}
\begin{multline*}
 = \left(\frac{Q+a-4}{2}\right)^{2}\int_{0}^{1} \int_{\wp}\left|\R_{r} f\right|^{2} r^{Q-a+1}d\sigma(y) d r\\\geq\left(\frac{Q+a-4}{2}\right)^{2}\left(\frac{c}{2}\right)^{2} \int_{0}^{1} \int_{\wp}|f|^{2} r^{Q-a-1}\left(1-r^{c}\right)^{-2} d\sigma(y) d r 
\\=\left(\frac{(Q+a-4)c}{4}\right)^{2} \int_{B(0,1)} \frac{|f|^{2}}{|x|^{a}\left(1-|x|^{c}\right)^{2}} d x
\end{multline*}
yielding \eqref{(24)} for all $f \in C_{0}^{\infty}\left(B(0,1) \backslash\{0\}\right)$, where we have used one-dimensional Hardy inequality \cite[Proposition 4.3]{San21} in the last inequality. 

To obtain \eqref{(25)}, for $f \in C_{0}^{\infty}\left(B(0,1) \backslash\{0\}\right)$ and $3\leq a \leq  \min \{Q-c+2, Q-3 c\}$, it follows from \eqref{eq_1.14_1} and \cite[Proposition 4.3]{San21} that
$$
\begin{aligned}
\int_{B(0,1)} \frac{\left|\mathcal{R}_{|x|}^{2}f+\frac{Q-1}{|x|} \mathcal{R}_{|x|} f\right|^{2}}{|x|^{a-4}} d x & \geq \int_{0}^{1} \int_{\wp}\left|\R_{r}f\right|^{2} r^{Q-a+3}d\sigma(y) d r\\
& \geq\left(\frac{c}{2}\right)^{2} \int_{0}^{1} \int_{\wp}\left|\R_{r}f\right|^{2} r^{Q-a+1}\left(1-r^{c}\right)^{-2} d\sigma(y) d r \\
& \geq\left(\frac{3}{4} c^{2}\right)^{2} \int_{0}^{1} \int_{\wp}|f|^{2} r^{Q-a-1}\left(1-r^{c}\right)^{-4}d\sigma(y) d r\\&=\left(\frac{3}{4} c^{2}\right)^{2} \int_{B(0,1)} \frac{|f|^{2}}{|x|^{a}\left(1-|x|^{c}\right)^{4}} d x.
\end{aligned}
$$
The optimality of the constants in \eqref{(24)} and \eqref{(25)} can be shown as in the Euclidean case (see \cite[Theorem 3.1]{San21}). 

\end{proof}
\begin{proof}[Proof of Theorem \ref{thm_Lp_Rel}] Using Theorem \ref{T IH remainder} with $b=p$ and \eqref{(38)} we obtain
\begin{multline*}
\left(\frac{p-1}{p} c\right)^{p} \int_{B(0,R)} \frac{|f|^{p}}{|x|^{a}\left(1-\left(\frac{| x|}{R}\right)^{c}\right)^{p}}dx+\psi_{Q, p, a, p}(f) \leq \int_{B(0,R)} \frac{|\R_{|x|} f|^{p}}{|x|^{a-p}} d x \\
\leq\left|\frac{Q(p-1)+a-2 p}{p}\right|^{-p} \int_{B(0,R)} \frac{\left|\mathcal{R}_{|x|}^{2}f+\frac{Q-1}{|x|} \mathcal{R}_{|x|} f\right|^{p}}{|x|^{a-2 p}} d x
\end{multline*}
yielding \eqref{(35)}. The optimality of the constant in \eqref{(35)} without the remainder term can be shown as in the Euclidean case (see \cite[Theorem 3.6]{San21}). 
\end{proof}


\begin{thebibliography}{BFKG12}

\bibitem[ACFM23]{ACFM22}
E.A.M.~Abreu, J.L.~Carvalho, M.F.~Furtado and E.S.~Medeiros.
\newblock {\em Hardy inequality for domains with a geometric boundary condition and applications.}
\newblock Bull. Lond. Math. Soc., 55(1):428--446, 2023.

\bibitem[BG84]{BG84}
P.~Baras and J.~A.~Goldstein.
\newblock {\em The heat equation with a singular potential.}
\newblock Trans. Amer. Math. Soc., 284:121--139, 1984.

\bibitem[BFL08]{BFL08}
R.~D.~Benguria, R.~L.~Frank and M.~Loss.
\newblock {\em The sharp constant in the Hardy-Sobolev-Maz'ya inequality
in the three dimensional upper half-space.}
\newblock Math. Res. Lett., 15(4):613--622, 2008.

\bibitem[BFT03]{BFT03}
G.~Barbatis, S.~Filippas and A.~Tertikas.
\newblock {\em Series expansion for $L^p$ Hardy inequalities.}
\newblock Indiana Univ.
Math. J., 52(1):171--190, 2003.

\bibitem[BM97]{BM97}
H.~Brezis and M.~Marcus.
\newblock  {\em Hardy’s inequalities revisited.}
\newblock Ann. Scuola Norm. Sup. Pisa Cl. Sci. (4), 25(1-2):217--237, 1997.

\bibitem[BMO23]{BMO23}
N.~Bez, S.~Machihara and T.~Ozawa.
\newblock  {\em Revisiting the Rellich inequality.}
\newblock Math. Z., 303: Art. No. 49, 2023.

\bibitem[BT06]{BT06}
G.~Barbatis and A.~Tertikas.
\newblock {\em On a class of Rellich inequalities.}
\newblock J. Comput. Appl. Math., 194(1):156--172, 2006.

\bibitem[BV97]{BV97}
H.~Brezis and J.~L.~V\'{a}zquez.
\newblock {\em Blow-up solutions of some nonlinear elliptic problems.}
\newblock Rev. Mat. Univ. Complut. Madrid, 10(2):443--469, 1997.


\bibitem[CM12]{CM12}
P.~Caldiroli and R.~Musina.
\newblock {\em Rellich inequalities with weights.}
\newblock Calc. Var. Partial Differ.
Equ., 45(1-2):147--164, 2012.

\bibitem[Amb05]{Amb05}
L.~D'Ambrosio.
\newblock {\em Hardy-type inequalities related to degenerate elliptic differential operators.}
\newblock Ann.
Sc. Norm. Super. Pisa Cl. Sci. (5), 4:451--486, 2005.

\bibitem[FS82]{FS-book}
G.~B. Folland and E.~M. Stein.
\newblock {\em Hardy spaces on homogeneous groups}, volume~28 of {\em
  Mathematical Notes}.
\newblock Princeton University Press, Princeton, N.J.; University of Tokyo
  Press, Tokyo, 1982.

  \bibitem[FS08]{FS08}
R. L.~Frank and R.~Seiringer.
\newblock {\em Non-linear ground state representations and sharp Hardy inequalities.}
\newblock J. Funct. Anal., 255:3407--3430, 2008.

\bibitem[Iok19]{Iok19}
N.~Ioku.
\newblock {\em Attainability of the best Sobolev constant in a ball.}
\newblock Math. Ann., 375(1-2):1--16, 2019.

\bibitem[IIO17]{IIO17}
N.~Ioku, M.~Ishiwata and T.~Ozawa.
\newblock {\em Hardy type inequalities in $L_p$
 with sharp remainders.}
\newblock J. Inequal. Appl., 2017: Art. No. 5, 2017.

\bibitem[Lin90]{Lin90}
P.~Lindqvist.
\newblock {\em On the equation $\operatorname{div}\left(|\nabla u|^{p-2} \nabla u\right)+\lambda|u|^{p-2} u=0$.}
\newblock Proc. Amer. Math. Soc., 109(1):157--164, 1990.

\bibitem[MOW15]{MOW15}
S.~Machihara, T.~Ozawa, and H.~Wadade.
\newblock {\em Scaling invariant Hardy inequalities of multiple
logarithmic type on the whole space.}
\newblock J. Inequal. Appl., 281:1--13, 2015.

\bibitem[MOW17]{MOW17}
S.~Machihara, T.~Ozawa, and H.~Wadade.
\newblock {\em Remarks on the Rellich inequality.}
\newblock Math. Z., 286:1367--1373, 2017.

\bibitem[Maz11]{Maz11}
V.~Maz'ya.
\newblock {\em Sobolev spaces with applications to elliptic partial differential equations.}
\newblock Second, revised
and augmented edition, Grundlehren der Mathematischen Wissenschaften [Fundamental Principles
of Mathematical Sciences], 342. Springer, Heidelberg, 2011. xxviii+866 pp.

\bibitem[Ngu17]{Ngu17}
V.H.~Nguyen.
\newblock {\em The sharp higher order Hardy–Rellich type inequalities
on the homogeneous groups.}
\newblock arXiv preprint arXiv:1708.09311, 2017.

\bibitem[RS17]{RS17}
M.~Ruzhansky and D.~Suragan.
\newblock {\em Hardy and Rellich inequalities, identities, and sharp remainders on homogeneous groups.}
\newblock Adv. Math., 317:799--822, 2017.

\bibitem[RS19]{RS_book}
M.~Ruzhansky and D.~Suragan.
\newblock {\em Hardy inequalities on homogeneous groups}, volume 537 of Progress in Mathematics.
\newblock Birkh\"auser, 2019. 

\bibitem[RSY18a]{RSY18}
M.~Ruzhansky, D.~Suragan and N.~Yessirkegenov.
\newblock {\em Sobolev type inequalities, Euler-Hilbert-Sobolev and Sobolev-Lorentz-Zygmund spaces on homogeneous groups.}
\newblock Integral Equ. Oper. Theory, 90:10, 2018.

\bibitem[RSY18b]{RSY18_Tran}
M.~Ruzhansky, D.~Suragan and N.~Yessirkegenov,
\newblock {\em Extended Caffarelli-Kohn-Nirenberg inequalities, and remainders, stability, and superweights for $L_{p}$-weighted Hardy inequalities},
\newblock Trans. Amer. Math. Soc. Ser. B, 5:32--62, 2018.

\bibitem[RSY17]{RSY17_Comp}
M.~Ruzhansky, D.~Suragan and N.~Yessirkegenov, \newblock {\em Extended Caffarelli–Kohn–Nirenberg inequalities and superweights for Lp-weighted Hardy inequalities}.
\newblock Comptes Rendus Mathematique,  355(6): 694--698, 2017. 

\bibitem[San22]{San21}
M.~Sano.
\newblock {\em Improvements and generalizations of two Hardy type inequalities and their applications to the Rellich type inequalities}.
\newblock Milan J. Math., 90(1):647--678, 2022.

\bibitem[Tak15]{Tak15}
F.~Takahashi.
\newblock {\em A simple proof of Hardy's inequality in a limiting case.}
\newblock Archiv der Math., 104(1):77--82, 2015.

\bibitem[TT07]{TT07}
A.~Tertikas and K.~Tintarev.
\newblock {\em On existence of minimizers for the Hardy-Sobolev-Maz'ya inequality.}
\newblock Ann.
Mat. Pura Appl. (4), 186(4):645--662, 2007.

\bibitem[Yan13]{Yan13}
Q.~Yang.
\newblock {\em Hardy type inequalities related to Carnot-Carathéodory distance on the Heisenberg group.}
\newblock Proc. Amer. Math. Soc., 141:351--362, 2013.
\end{thebibliography}
\end{document}